\documentclass{amsart}
\usepackage{amssymb}
\usepackage{amsfonts}
\usepackage{geometry}

\newtheorem{theorem}{Theorem}
\theoremstyle{plain}

\newtheorem{conjecture}[theorem]{Conjecture}

\newtheorem{definition}[theorem]{Definition}

\newtheorem{lemma}[theorem]{Lemma}

\newtheorem{proposition}[theorem]{Proposition}
\newtheorem{remark}[theorem]{Remark}

\numberwithin{equation}{section}
\input{tcilatex}
\begin{document}
\title{Weighted Alpert wavelets}
\author{Rob Rahm}
\author{Eric T. Sawyer$^\dagger$}
\thanks{$\dagger $ Research supported in part by a grant from the National
Science and Engineering Research Council of Canada.}
\author{Brett D. Wick$^\ddagger$}
\thanks{$\ddagger $ Research supported in part by National Science
Foundation DMS grants \# 1560955 and 1800057.}
\date{\today }

\begin{abstract}
In this paper we construct a wavelet basis in $L^{2}(\mathbb{R}^{n};\mu )$
possessing vanishing moments of a fixed order for a general locally finite
positive Borel measure $\mu $. The approach is based on a clever
construction of Alpert in the case of Lebesgue measure that is appropriately
modified to handle the general measures considered here. We then use this
new wavelet basis to study a two-weight inequality for a general Calder\'{o}%
n-Zygmund operator on $\mathbb{R}$ and conjecture that under suitable
natural conditions, including a \emph{weaker} energy condition, the operator
is bounded from $L^{2}(\mathbb{R};\sigma )$ to $L^{2}(\mathbb{R};\omega )$
if certain \emph{stronger} testing conditions hold on polynomials. An
example is provided showing that this conjecture is logically different than
existing results in the literature.
\end{abstract}

\maketitle
\tableofcontents

\section{Introduction and statement of main results}

The use of weighted Haar wavelet expansions has its roots in connection with
the $Tb$ theorem in \cite{DaJoSe} and \cite{CoJoSe}, and came to fruition in
treating the two weight norm inequality for the Hilbert transform in \cite%
{NTV4}, \cite{Vol}, the two part paper \cite{LaSaShUr3},\cite{Lac} and \cite%
{Hyt2}. The key features of the weighted Haar expansion $\left\{ h_{I}^{\mu
}\right\} _{I\in \mathcal{D}}$ are threefold:

\begin{enumerate}
\item The Haar functions $\left\{ h_{I}^{\mu }\right\} _{I\in \mathcal{D}}$
form an orthonormal basis of $L^{2}\left( \mu \right) $: 
\begin{eqnarray*}
&&f=\sum_{I\in \mathcal{D}}\left\langle f,h_{I}^{\mu }\right\rangle
_{L^{2}\left( \mu \right) }h_{I}^{\mu }\text{ both pointwise }\mu \text{%
-a.e. and in }L^{2}\left( \mu \right) , \\
&&\text{where }\left\langle h_{J}^{\mu },h_{I}^{\mu }\right\rangle
_{L^{2}\left( \mu \right) }=\delta _{I}^{J},
\end{eqnarray*}

\item Telescoping identities hold: 
\begin{equation*}
\mathbf{1}_{K}\sum_{I\in \mathcal{D}:\ K\subsetneqq I\subset L}\left\langle
f,h_{I}^{\mu }\right\rangle _{L^{2}\left( \mu \right) }h_{I}^{\mu
}=E_{K}^{\mu }f-E_{L}^{\mu }f,\ \ \ \ \ K\subsetneqq L,
\end{equation*}

\item Moment vanishing conditions hold:%
\begin{equation*}
\int h_{I}^{\mu }\left( x\right) d\mu \left( x\right) =0,\ \ \ \ \ I\in 
\mathcal{D}.
\end{equation*}
\end{enumerate}

In the setting of Lebesgue measure, Alpert \cite{Alp} introduced new
wavelets with more vanishing moments in (3), while retaining orthonormality
(1) and telescoping (2). The expense of imposing these extra moment
conditions is that one requires additional functions in order to obtain the
expansion. The purpose of this note is to extend existence of Alpert
wavelets to arbitrary locally finite positive Borel measures in Euclidean
space $\mathbb{R}^{n}$, and to investigate degeneracy and uniqueness in the
one-dimensional case as well. To state the main result in this paper
requires some notation.

Let $\mu $ be a locally finite positive Borel measure on $\mathbb{R}^{n}$,
and fix $k\in \mathbb{N}$. For $Q\in \mathcal{P}^{n}$, the collection of
cubes with sides parallel to the coordinate axes, denote by $%
L_{Q;k}^{2}\left( \mu \right) $ the finite dimensional subspace of $%
L^{2}\left( \mu \right) $ that consists of linear combinations of the
indicators of\ the children $\mathfrak{C}\left( Q\right) $ of $Q$ multiplied
by polynomials of degree at most $k-1$, and such that the linear
combinations have vanishing $\mu $-moments on the cube $Q$ up to order $k-1$:%
\begin{equation*}
L_{Q;k}^{2}\left( \mu \right) \equiv \left\{ f=\dsum\limits_{Q^{\prime }\in 
\mathfrak{C}\left( Q\right) }\mathbf{1}_{Q^{\prime }}p_{Q^{\prime };k}\left(
x\right) :\int_{Q}f\left( x\right) x_{i}^{\ell }d\mu \left( x\right) =0,\ \
\ \text{for }0\leq \ell \leq k-1\text{ and }1\leq i\leq n\right\} ,
\end{equation*}%
where $p_{Q^{\prime };k}\left( x\right) =\sum_{\alpha \in \mathbb{Z}%
_{+}^{n}:\left\vert \alpha \right\vert \leq k-1\ }a_{Q^{\prime };\alpha
}x^{\alpha }$ is a polynomial in $\mathbb{R}^{n}$ of degree $\left\vert
\alpha \right\vert =\alpha _{1}+...+\alpha _{n}$ at most $k-1$. Here $%
x^{\alpha }=x_{1}^{\alpha _{1}}x_{2}^{\alpha _{2}}...x_{n}^{\alpha _{n}}$.
Let $d_{Q;k}\equiv \dim L_{Q;k}^{2}\left( \mu \right) $ be the dimension of
the finite dimensional linear space $L_{Q;k}^{2}\left( \mu \right) $.

Now define%
\begin{eqnarray*}
&&\mathcal{F}_{\infty }^{k}\left( \mu \right) \equiv \left\{ \alpha \in 
\mathbb{Z}_{+}^{n}:\left\vert \alpha \right\vert \leq k-1:x^{\alpha }\in
L^{2}\left( \mu \right) \right\} \ , \\
&&\ \ \ \ \ \ \ \ \ \ \ \ \ \ \ \text{and }\mathcal{P}_{\mathbb{R}%
^{n}}^{k}\left( \mu \right) \equiv \limfunc{Span}\left\{ x^{\alpha }\right\}
_{\alpha \in \mathcal{F}_{\infty }^{k}}\ .
\end{eqnarray*}%
Let $\bigtriangleup _{Q;k}^{\mu }$ denote orthogonal projection onto the
finite dimensional subspace $L_{Q;k}^{2}\left( \mu \right) $, let $\mathbb{E}%
_{Q;k}^{\mu }$ denote orthogonal projection onto the finite dimensional
subspace%
\begin{equation*}
\mathnormal{\limfunc{Span}}\{\mathbf{1}_{Q}\left( x\right) x^{\alpha }:0\leq
\left\vert \alpha \right\vert \leq k-1\},
\end{equation*}%
and let $\bigtriangleup _{\mathbb{R}^{n};k}^{\mu }$ denote orthogonal
projection onto $\mathcal{P}_{\mathbb{R}^{n}}^{k}\left( \mu \right) $. The
projections $\bigtriangleup _{Q;k}^{\mu }$ are often referred to as
multiresolution projections.

The first of two main results proved in this note is the following theorem,
which establishes the existence of Alpert wavelets in all dimensions having
the three important properties of orthogonality, telescoping and moment
vanishing.

\begin{theorem}[Weighted Alpert Bases]
\label{main1}Let $\mu $ be a locally finite positive Borel measure on $%
\mathbb{R}^{n}$, fix $k\in \mathbb{N}$, and fix a dyadic grid $\mathcal{D}$
in $\mathbb{R}^{n}$.

\begin{enumerate}
\item Then $\left\{ \bigtriangleup _{\mathbb{R}^{n};k}^{\mu }\right\} \cup
\left\{ \bigtriangleup _{Q;k}^{\mu }\right\} _{Q\in \mathcal{D}}$ is a
complete set of orthogonal projections in $L_{\mathbb{R}^{n}}^{2}\left( \mu
\right) $ and%
\begin{eqnarray*}
f &=&\bigtriangleup _{\mathbb{R}^{n};k}^{\mu }f+\sum_{Q\in \mathcal{D}%
}\bigtriangleup _{Q;k}^{\mu }f,\ \ \ \ \ f\in L_{\mathbb{R}^{n}}^{2}\left(
\mu \right) , \\
&&\left\langle \bigtriangleup _{\mathbb{R}^{n};k}^{\mu }f,\bigtriangleup
_{Q;k}^{\mu }f\right\rangle =\left\langle \bigtriangleup _{P;k}^{\mu
}f,\bigtriangleup _{Q;k}^{\mu }f\right\rangle =0\text{ for }P\neq Q,
\end{eqnarray*}%
where convergence in the first line holds both in $L_{\mathbb{R}%
^{n}}^{2}\left( \mu \right) $ norm and pointwise $\mu $-almost everywhere.

\item Moreover we have the telescoping identities%
\begin{equation}
\mathbf{1}_{Q}\sum_{Q\subsetneqq I\subset P}\bigtriangleup _{I;k}^{\mu }=%
\mathbb{E}_{Q;k}^{\mu }-\mathbb{E}_{P;k}^{\mu }\ \text{ \ for }P,Q\in 
\mathcal{D}\text{ with }P\subsetneqq Q,  \label{telescoping}
\end{equation}

\item and the moment conditions%
\begin{equation}
\int_{\mathbb{R}^{n}}\bigtriangleup _{Q;k}^{\mu }f\left( x\right) x^{\alpha
}d\mu \left( x\right) =0,\ \ \ \text{for }Q\in \mathcal{D},\text{ }0\leq
\left\vert \alpha \right\vert \leq k-1.  \label{mom con}
\end{equation}
\end{enumerate}
\end{theorem}

In the special case of dimension $n=1$, we further investigate uniqueness
and degeneracy of the wavelets constructed in Theorem \ref{main1}. The
system of one-dimensional Alpert wavelets is underdetermined, in general
having $\left( 
\begin{array}{c}
k \\ 
2%
\end{array}%
\right) $ additional degrees of freedom which can be used to impose
additional moment conditions in (3). The degeneracy condition for Alpert
wavelets is phrased in terms of a matrix of moments $\boldsymbol{M}_{Q,k}$
and positive semi-definiteness, and can be interpreted as the degree to
which $\mu $ is a finite sum of point mass measures within a given child of
a cube. Here is our second main result, which includes the main points of
our investigation into uniqueness and degeneracy, but not all of them - see
Subsection \ref{Sub special} below for more. Let%
\begin{equation*}
\boldsymbol{M}_{Q,k}=\int_{Q}\left[ 
\begin{array}{ccccc}
1 & x & \cdots & x^{k-2} & x^{k-1} \\ 
x & x^{2} & \ddots & x^{k-3} & x^{k-2} \\ 
\vdots & \ddots & \ddots & \ddots & \vdots \\ 
x^{k-2} & x^{k-3} & \ddots & x^{2k-4} & x^{2k-3} \\ 
x^{k-1} & x^{k-2} & \cdots & x^{2k-3} & x^{2k-2}%
\end{array}%
\right] d\mu \left( x\right)
\end{equation*}%
be the symmetric matrix of moments of the measure $\mu $ up to order $k-1$
on the interval $Q$.

\begin{theorem}
\label{dim one}Let $\mu $ be a locally finite positive Borel measure on $%
\mathbb{R}$, fix $k\in \mathbb{N}$, and fix a dyadic grid $\mathcal{D}$ in $%
\mathbb{R}$. Then, in addition to parts (1) and (2) of Theorem \ref{main1}
(restricted to $n=1$), we also have:

\begin{enumerate}
\item The dimension of $L_{Q;k}^{2}\left( \mu \right) $ is given by%
\begin{equation*}
\dim L_{Q;k}^{2}\left( \mu \right) =\dim \left( \limfunc{Range}\boldsymbol{L}%
_{Q,k}\bigcap \limfunc{Range}\boldsymbol{R}_{Q,k}\right) \ ,
\end{equation*}%
where $\boldsymbol{L}_{Q,k}$ and $\boldsymbol{R}_{Q,k}$ denote the matrices $%
\boldsymbol{M}_{Q_{\limfunc{left}},k}$ and $\boldsymbol{M}_{Q_{\limfunc{right%
}},k}$ respectively and $Q_{\limfunc{left}}/Q_{\limfunc{right}}$ are the
left and right halves of the interval $Q$. This shows in particular that $%
L_{Q;k}^{2}\left( \mu \right) =k$ for all dyadic intervals $Q$ if and only
if $\boldsymbol{M}_{Q,k}\succ 0$ for all dyadic intervals $Q$.

\item In the case when $\boldsymbol{M}_{Q,k}\succ 0$ for all dyadic
intervals $Q$, we can choose an orthonormal basis $\left\{ a_{Q}^{\mu ,\ell
}\right\} _{\ell =1}^{k}$ of $L_{Q;k}^{2}\left( \mu \right) $ so that in
addition to the moment conditions given in part (3) of Theorem \ref{main1},
the following $\emph{additional}$ moment conditions hold:%
\begin{equation*}
\int a_{Q}^{\mu ,\ell }\left( x\right) x^{i}d\mu \left( x\right) =0,\ \ \ \
\ \text{for all }2\leq \ell \leq k\text{ and }k\leq i\leq k+\ell -2.
\end{equation*}
\end{enumerate}
\end{theorem}

The additional moment conditions in part (2) consume the remaining $\left( 
\begin{array}{c}
k \\ 
2%
\end{array}%
\right) $ degrees of freedom available in defining the $2k^{2}$ coefficients
in the functions $\left\{ a_{Q}^{\mu ,\ell }\right\} _{\ell =1}^{k}$, which
we refer to as Alpert functions.

\begin{remark}
There is an analogous theorem in higher dimensions $n>1$, whose formulation
and proof we leave for the interested reader.
\end{remark}

When $k=1$, these theorems reduce to the well-known weighted Haar basis in $%
L^{2}(\mu )$ which is recalled in detail in Section 2. However, when $k\geq 2
$, new wavelet bases are provided by this construction, and in Section 3, we
prove our two main results regarding these new bases, Theorems \ref{main1}
and \ref{dim one}. In the final section of this paper, we use the Alpert
basis to study weighted inequalities for Calder\'{o}n-Zygmund operators on
the real line. A natural proof strategy is to decompose $f\in L^{2}(\sigma )$
and $g\in L^{2}(\omega )$ via a wavelet basis and then analyze the behavior
of the matrix of $T$ relative to this basis. One then arrives at testing
conditions related to the number of moments that the basis possess, as well
as a relaxation of the energy condition required to control certain terms in
the matrix. With this new basis we study a $Tp$ type result in the two
weight setting. See Conjecture \ref{TP} below.

\section{Weighted Haar bases}

In this section we review the existence, uniqueness and degeneracy of the
weighted Haar wavelets, beginning with the local case. Let $\mu $ be a
locally finite positive Borel measure on the unit interval $\left[
0,1\right) $.

\begin{definition}
Set $I_{0}^{0}\equiv \left[ 0,1\right) $, $I_{0}^{1}\equiv \left[ 0,\frac{1}{%
2}\right) $, $I_{1}^{1}\equiv \left[ \frac{1}{2},1\right) $, and in general%
\begin{equation*}
I_{j}^{m}\equiv \left[ \frac{j}{2^{m}},\frac{j+1}{2^{m}}\right) ,\ \ \ \ \
0\leq j\leq 2^{m}-1,\ m\in \mathbb{Z}_{+}\ .
\end{equation*}
\end{definition}

Note that the left half of the interval $I_{j}^{m}$ is $I_{j,\limfunc{left}%
}^{m}=I_{2j}^{m+1}$ and that the right half of the interval $I_{j}^{m}$ is $%
I_{j,\limfunc{right}}^{m}=I_{2j+1}^{m+1}$. We begin by briefly reviewing the
weighted Haar wavelet bases on the real line and in Euclidean space, and in
the next section we will turn to weighted Alpert wavelet bases with more
vanishing moments.

Suppose that $\mu $ is a finite positive Borel measure on $\left[ 0,1\right) 
$ and define 
\begin{equation*}
L^{2}\left( \mu \right) \equiv \left\{ f:\left[ 0,1\right) \rightarrow 
\mathbb{C}\ \mu \text{-measurable}:\int \left\vert f\right\vert ^{2}d\mu
<\infty \right\} .
\end{equation*}%
Let $\left\vert E\right\vert _{\mu }$ denote the $\mu $-measure of a Borel
set $E$. Assume first the nondegeneracy condition that $\mu $ charges every
dyadic subinterval of $\left[ 0,1\right) $, 
\begin{equation}
\left\vert J\right\vert _{\mu }>0\text{ for all for all }J=I_{j}^{m},\ m\in 
\mathbb{Z}_{+},\ 0\leq j\leq 2^{m}-1.  \label{Haar nondeg}
\end{equation}

\begin{definition}
For all $m\in \mathbb{Z}_{+}$ and $0\leq j\leq 2^{m}-1$, define%
\begin{equation}
h_{I_{j}^{m}}^{\mu }\left( x\right) \equiv -\sqrt{\frac{1}{\left\vert
I_{j}^{m}\right\vert _{\mu }}}\sqrt{\frac{\left\vert I_{j,\limfunc{right}%
}^{m}\right\vert _{\mu }}{\left\vert I_{j,\limfunc{left}}^{m}\right\vert
_{\mu }}}\mathbf{1}_{I_{j,\limfunc{left}}^{m}}+\sqrt{\frac{1}{\left\vert
I_{j}^{m}\right\vert _{\mu }}}\sqrt{\frac{\left\vert I_{j,\limfunc{left}%
}^{m}\right\vert _{\mu }}{\left\vert I_{j,\limfunc{right}}^{m}\right\vert
_{\mu }}}\mathbf{1}_{I_{j,\limfunc{right}}^{m}}\ .  \label{simple formula}
\end{equation}%
Finally define%
\begin{equation*}
h_{0}^{\mu }\left( x\right) \equiv \sqrt{\frac{1}{\left\vert \left[
0,1\right) \right\vert _{\mu }}}\mathbf{1}_{\left[ 0,1\right) }\left(
x\right) .
\end{equation*}
\end{definition}

\begin{theorem}
The collection 
\begin{equation*}
\mathcal{U}_{\left[ 0,1\right) }^{\limfunc{Haar},\mu }\equiv \left\{
h_{0}^{\mu }\right\} \cup \left\{ h_{I_{j}^{m}}^{\mu }\right\} _{m\in 
\mathbb{Z}_{+}\text{ and }0\leq j\leq 2^{m}-1}
\end{equation*}%
is an orthonormal basis for $L^{2}\left( \mu \right) $.
\end{theorem}

\begin{proof}
It is a straightforward computation to see that $\mathcal{U}_{\left[
0,1\right) }^{\limfunc{Haar},\mu }$ is an orthonormal set in $L^{2}\left(
\mu \right) $, and the dyadic Lebesgue differentiation theorem with respect
to the measure $\mu $, together with the multi-resolution telescoping
identities, show that $\mathcal{U}_{\left[ 0,1\right) }^{\limfunc{Haar},\mu
} $ is complete in $L^{2}\left( \mu \right) $.
\end{proof}

\subsection{Derivation of the explicit formula $\left( \protect\ref{simple
formula}\right) $}

The coefficients on the functions $\mathbf{1}_{I_{j,\limfunc{left}}^{m}}$
and $\mathbf{1}_{I_{j,\limfunc{right}}^{m}}$ are derived in the following
way. If we set 
\begin{equation*}
h_{I_{j}^{m}}\left( x\right) \equiv -\alpha \mathbf{1}_{I_{j,\limfunc{left}%
}^{m}}+\beta I_{j,\limfunc{right}}^{m},
\end{equation*}%
and demand that both%
\begin{eqnarray*}
0 &=&\int_{I_{j}^{m}}h_{I_{j}^{m}}d\mu =-\alpha \left\vert I_{j,\limfunc{left%
}}^{m}\right\vert _{\mu }+\beta \left\vert I_{j,\limfunc{right}%
}^{m}\right\vert _{\mu }\ , \\
1 &=&\int_{I_{j}^{m}}\left\vert h_{I_{j}^{m}}\right\vert ^{2}d\mu =\alpha
^{2}\left\vert I_{j,\limfunc{left}}^{m}\right\vert _{\mu }+\beta
^{2}\left\vert I_{j,\limfunc{right}}^{m}\right\vert _{\mu },
\end{eqnarray*}%
then we must have%
\begin{equation*}
\beta =\frac{\left\vert I_{j,\limfunc{left}}^{m}\right\vert _{\mu }}{%
\left\vert I_{j,\limfunc{right}}^{m}\right\vert _{\mu }}\alpha \text{ and }%
1=\alpha ^{2}\left\vert I_{j,\limfunc{left}}^{m}\right\vert _{\mu }+\left( 
\frac{\left\vert I_{j,\limfunc{left}}^{m}\right\vert _{\mu }}{\left\vert
I_{j,\limfunc{right}}^{m}\right\vert _{\mu }}\right) ^{2}\alpha
^{2}\left\vert I_{j,\limfunc{right}}^{m}\right\vert _{\mu }=\alpha
^{2}\left\vert I_{j}^{m}\right\vert _{\mu }\frac{\left\vert I_{j,\limfunc{%
left}}^{m}\right\vert _{\mu }}{\left\vert I_{j,\limfunc{right}%
}^{m}\right\vert _{\mu }},
\end{equation*}%
which implies%
\begin{equation*}
\alpha =\sqrt{\frac{\left\vert I_{j,\limfunc{right}}^{m}\right\vert _{\mu }}{%
\left\vert I_{j}^{m}\right\vert _{\mu }\left\vert I_{j,\limfunc{left}%
}^{m}\right\vert _{\mu }}}\text{ and }\beta =\sqrt{\frac{\left\vert I_{j,%
\limfunc{left}}^{m}\right\vert _{\mu }}{\left\vert I_{j}^{m}\right\vert
_{\mu }\left\vert I_{j,\limfunc{right}}^{m}\right\vert _{\mu }}},
\end{equation*}%
where we see that $\alpha $ and $\beta $ are uniquely determined up to sign.

\subsection{The Haar degenerate case}

Here we examine what happens when the nondegeneracy condition (\ref{Haar
nondeg}) fails.

\begin{lemma}
If at least one of the two children of $I_{j}^{m}$ fails to be charged by $%
\mu $, then if $h_{I_{j}^{m}}$ is defined according to the derivation
outlined above, we have that $h_{I_{j}^{m}}\equiv 0$.
\end{lemma}

\begin{proof}
Fix $I_{j}^{m}$ with $m\in \mathbb{Z}_{+}$ and $0\leq j\leq 2^{m}-1$. If $%
\left\vert I_{j}^{m}\right\vert _{\mu }=0$, then clearly $h_{I_{j}^{m}}=0$
in $L^{2}\left( \mu \right) $. If just one of the children of $I_{j}^{m}$ is
not charged by $\mu $, say $\left\vert I_{j,\limfunc{left}}^{m}\right\vert
_{\mu }=0$ but $\left\vert I_{j,\limfunc{right}}^{m}\right\vert _{\mu }>0$,
then the moment requirement above becomes $0=-\alpha \left\vert I_{j,%
\limfunc{left}}^{m}\right\vert _{\mu }+\beta \left\vert I_{j,\limfunc{right}%
}^{m}\right\vert _{\mu }=\beta \left\vert I_{j,\limfunc{right}%
}^{m}\right\vert _{\mu }$, which implies $\beta =0$, and hence $%
h_{I_{j}^{m}}=-\alpha \mathbf{1}_{I_{j,\limfunc{left}}^{m}}+\beta I_{j,%
\limfunc{right}}^{m}=-\alpha \mathbf{1}_{I_{j,\limfunc{left}}^{m}}\equiv 0$
in $L^{2}\left( \mu \right) $. Thus we see that $h_{I_{j}^{m}}\equiv 0$ if
at least one of its children is not charged by $\mu $.
\end{proof}

The resulting pared collection $\mathcal{U}_{\left[ 0,1\right) }^{\limfunc{%
Haar},\mu }\equiv \left\{ h_{0}^{\mu }\right\} \cup \left\{
h_{I_{j}^{m}}^{\mu }\right\} _{m\in \mathbb{Z}_{+}\text{ and }0\leq j\leq
2^{m}-1}$, where now $h_{I_{j}^{m}}^{\mu }$ is removed if it vanishes
identically, is still an orthonormal basis for $L^{2}\left( \mu \right) $.
Indeed, this follows from the fact that the telescoping identities still
hold, and hence Lebesgue's dyadic differentiation theorem continues to show
the pared collection $\mathcal{U}_{\left[ 0,1\right) }^{\limfunc{Haar},\mu }$
is complete in $L^{2}\left( \mu \right) $.

\subsection{The global case}

If $\mu $ is a locally finite positive Borel measure on the real line $%
\mathbb{R}$, $\mathcal{D}$ is a dyadic grid in $\mathbb{R}$, and%
\begin{equation*}
h_{I}^{\mu }\left( x\right) \equiv \left\{ 
\begin{array}{ccc}
-\sqrt{\frac{1}{\left\vert I\right\vert _{\mu }}}\sqrt{\frac{\left\vert I_{%
\limfunc{right}}\right\vert _{\mu }}{\left\vert I_{\limfunc{left}%
}\right\vert _{\mu }}}\mathbf{1}_{I_{\limfunc{left}}}+\sqrt{\frac{1}{%
\left\vert I\right\vert _{\mu }}}\sqrt{\frac{\left\vert I_{\limfunc{left}%
}\right\vert _{\mu }}{\left\vert I_{\limfunc{right}}\right\vert _{\mu }}}%
\mathbf{1}_{I_{\limfunc{right}}} & \text{ if } & \min \left\{ \left\vert I_{%
\limfunc{right}}\right\vert _{\mu },\left\vert I_{\limfunc{left}}\right\vert
_{\mu }\right\} >0 \\ 
0 & \text{ if } & \min \left\{ \left\vert I_{\limfunc{right}}\right\vert
_{\mu },\left\vert I_{\limfunc{left}}\right\vert _{\mu }\right\} =0%
\end{array}%
\right.
\end{equation*}%
for each dyadic interval $I\in \mathcal{D}$, then the collection 
\begin{equation*}
\mathcal{U}_{\mathbb{R}}^{\limfunc{Haar},\mu }\equiv \left\{ h_{I}^{\mu
}\right\} _{I\in \mathcal{D}\ }
\end{equation*}%
is an orthonormal basis for $L^{2}\left( \mu \right) $ (where we of course
discard those $h_{I}^{\mu }$ that vanish identically).

\subsection{The higher dimensional case}

Again the local case $L^{2}\left( \left[ 0,1\right) ^{n};\mu \right) $ and
the global case $L^{2}\left( \mathbb{R}^{n};\mu \right) $ are treated
similarly and we only consider the global case $L^{2}\left( \mu \right)
=L^{2}\left( \mathbb{R}^{n};\mu \right) $ here. So suppose that $\mu $ is a
locally finite positive Borel measure on $\mathbb{R}^{n}$ and that $\mathcal{%
D}$ is a dyadic grid on $\mathbb{R}^{n}$. Given a dyadic cube $Q\in \mathcal{%
D}$ with $\left\vert Q\right\vert _{\mu }>0$, let $\bigtriangleup _{Q}^{\mu
} $ denote orthogonal projection onto the finite dimensional subspace $%
L_{Q;1}^{2}\left( \mu \right) $ of $L^{2}\left( \mu \right) $ that consists
of linear combinations of the indicators of\ the children $\mathfrak{C}%
\left( Q\right) $ of $Q$ that have $\mu $-mean zero over $Q$:%
\begin{equation*}
L_{Q;1}^{2}\left( \mu \right) \equiv \left\{ f=\dsum\limits_{Q^{\prime }\in 
\mathfrak{C}\left( Q\right) }a_{Q^{\prime }}\mathbf{1}_{Q^{\prime
}}:a_{Q^{\prime }}\in \mathbb{R},\int_{Q}fd\mu =0\right\} \\
=\mathnormal{\limfunc{Span}}\{\mathbf{1}_{Q^{\prime }}:Q^{\prime }\in 
\mathfrak{C}\left( Q\right) ,\left\vert Q^{\prime }\right\vert _{\mu
}>0\}\ominus \mathnormal{\limfunc{Span}}\{\mathbf{1}_{Q}\}.
\end{equation*}%
Thus, among other things, we see that $\dim L_{Q;1}^{2}=\#\{Q^{\prime }\in 
\mathfrak{C}(Q):\left\vert Q^{\prime }\right\vert _{\mu }>0\}-1$. If $%
\left\vert Q\right\vert _{\mu }=0$, set $\bigtriangleup _{Q}^{\mu }=0$. Then
we have the important telescoping property for dyadic cubes $Q_{1}\subset
Q_{2}$ (below $[Q_{1},Q_{2}]=\{Q:Q_{1}\subset Q\subset Q_{2}\}$):%
\begin{equation}
\mathbf{1}_{Q_{0}}\left( x\right) \left( \dsum\limits_{Q\in \left[
Q_{1},Q_{2}\right] }\bigtriangleup _{Q}^{\mu }f\left( x\right) \right) =%
\mathbf{1}_{Q_{0}}\left( x\right) \left( \mathbb{E}_{Q_{0}}^{\mu }f-\mathbb{E%
}_{Q_{2}}^{\mu }f\right) ,\ \ \ \ \ Q_{0}\in \mathfrak{C}\left( Q_{1}\right)
,\ f\in L^{2}\left( \mu \right) .  \label{telescope}
\end{equation}%
We will at times find it convenient to use a fixed orthonormal basis $%
\left\{ h_{Q}^{\mu ,a}\right\} _{a\in \Gamma _{n}}$ of $L_{Q}^{2}\left( \mu
\right) $ where $\Gamma _{n}$ is any convenient index set with cardinality
equal to the dimension of $L_{Q}^{2}\left( \mu \right) $, i.e. 
\begin{equation*}
\#\Gamma _{n}=\dim L_{Q}^{2}\left( \mu \right) =\#\left\{ Q^{\prime }\in 
\mathfrak{C}\left( Q\right) :\left\vert Q^{\prime }\right\vert _{\mu
}>0\right\} -1,
\end{equation*}%
where the second equality here follows from the fact that the functions in $%
L_{Q}^{2}\left( \mu \right) $ have vanishing mean. Then $\left\{ h_{Q}^{\mu
,a}\right\} _{a\in \Gamma _{n}\text{ and }Q\in \mathcal{D}}$ is an
orthonormal basis for $L^{2}\left( \mu \right) $, with the understanding
that we add the constant function $\mathbf{1}$ if $\mu $ is a finite
measure. In particular, if $\mu $ is an infinite measure, we have%
\begin{eqnarray*}
f\left( x\right) &=&\sum_{Q\in \mathcal{D}}\bigtriangleup _{Q}^{\mu }f\left(
x\right) ,\ \ \ \ \ \mu -a.e.x\in \mathbb{R}^{n}, \\
\left\Vert f\right\Vert _{L^{2}\left( \mu \right) }^{2} &=&\sum_{Q\in 
\mathcal{D}}\left\Vert \bigtriangleup _{Q}^{\mu }f\right\Vert _{L^{2}\left(
\mu \right) }^{2}=\sum_{Q\in \mathcal{D}}\sum_{a\in \Gamma _{n}}\left\vert 
\widehat{f}\left( Q\right) \right\vert ^{2},
\end{eqnarray*}%
where%
\begin{equation*}
\left\vert \widehat{f}\left( Q\right) \right\vert ^{2}\equiv \sum_{a\in
\Gamma _{n}}\left\vert \left\langle f,h_{Q}^{\mu ,a}\right\rangle _{\mu
}\right\vert ^{2},
\end{equation*}%
and the measure is suppressed in the notation. Indeed, this follows from (%
\ref{telescope}) and Lebesgue's differentiation theorem for cubes.

\section{Weighted Alpert wavelets with higher vanishing moments}

Let $\mu $ be a locally finite positive Borel measure on $\mathbb{R}^{n}$,
and fix $k\in \mathbb{N}$. In analogy with the definition of $%
L_{Q;1}^{2}\left( \mu \right) $ above for $Q\in \mathcal{P}^{n}$, we denote
by $L_{Q;k}^{2}\left( \mu \right) $ the finite dimensional subspace of $%
L^{2}\left( \mu \right) $ that consists of linear combinations of the
indicators of\ the children $\mathfrak{C}\left( Q\right) $ of $Q$ multiplied
by polynomials of degree at most $k-1$, and such that the linear
combinations have vanishing $\mu $-moments on the cube $Q$ up to order $k-1$:%
\begin{equation*}
L_{Q;k}^{2}\left( \mu \right) \equiv \left\{ f=\dsum\limits_{Q^{\prime }\in 
\mathfrak{C}\left( Q\right) }\mathbf{1}_{Q^{\prime }}p_{Q^{\prime };k}\left(
x\right) :\int_{Q}f\left( x\right) x_{i}^{\ell }d\mu \left( x\right) =0,\ \
\ \text{for }0\leq \ell \leq k-1\text{ and }1\leq i\leq n\right\} ,
\end{equation*}%
where $p_{Q^{\prime };k}\left( x\right) =\sum_{\alpha \in \mathbb{Z}%
_{+}^{n}:\left\vert \alpha \right\vert \leq k-1\ }a_{Q^{\prime };\alpha
}x^{\alpha }$ is a polynomial in $\mathbb{R}^{n}$ of degree $\left\vert
\alpha \right\vert =\alpha _{1}+...+\alpha _{n}$ at most $k-1$. Here $%
x^{\alpha }=x_{1}^{\alpha _{1}}x_{2}^{\alpha _{2}}...x_{n}^{\alpha _{n}}$.
Let $d_{Q;k}\equiv \dim L_{Q;k}^{2}\left( \mu \right) $ be the dimension of
the finite dimensional linear space $L_{Q;k}^{2}\left( \mu \right) $. Note
that the space $L_{Q;k}^{2}\left( \mu \right) $ can also be written as: 
\begin{equation}
\mathnormal{\limfunc{Span}}\{\mathbf{1}_{Q^{\prime }}x^{\alpha }:Q^{\prime
}\in \mathfrak{C}(Q),\left\vert Q^{\prime }\right\vert _{\mu }>0,\left\vert
\alpha \right\vert \leq k-1\}\ominus \left( \oplus _{\alpha :\left\vert
\alpha \right\vert \leq k-1}\mathbf{1}_{Q}x^{\alpha }\right) .
\label{also written}
\end{equation}

We begin with the proof of Theorem \ref{main1} in the next subsection below,
and we will complete the proof of Theorem \ref{dim one} in the third
subsection. In the final subsection we will give a complete and detailed
answer to both uniqueness and degeneracy in the special case when $n=1$ and $%
k=2$.

\subsection{Proof of Theorem \protect\ref{main1}}

We begin with an estimate of the dimension $d_{Q;k}$.

\begin{proposition}
\label{number}We have: 
\begin{equation*}
0\leq d_{Q;k}\leq \left( 2^{n}-1\right) A_{k,n}=\left( 2^{n}-1\right) \binom{%
n+k-1}{n},
\end{equation*}%
where $A_{k,n}$ denotes the number of non-negative integer solutions to $%
\alpha _{1}+\cdots +\alpha _{n}\leq k-1$.
\end{proposition}

\begin{proof}
First, if the functions $\{\mathbf{1}_{Q}(x)x^{\alpha }:\left\vert \alpha
\right\vert \leq k-1\}$ are linearly independent, then 
\begin{equation*}
\dim \mathnormal{\limfunc{Span}}\{\mathbf{1}_{Q}(x)x^{\alpha }:\left\vert
\alpha \right\vert \leq k-1\}=A_{k,n}.
\end{equation*}%
Now, let $B$ denote $\dim \mathnormal{\limfunc{Span}}\{\mathbf{1}%
_{Q}(x)x^{\alpha }:\left\vert \alpha \right\vert \leq k-1\}$. Then it
follows that 
\begin{equation*}
\dim \mathnormal{\limfunc{Span}}\{\mathbf{1}_{Q^{\prime }}(x)x^{\alpha
}:\left\vert \alpha \right\vert \leq k-1\}\leq B,\ \ \ \ \ \text{for all }%
Q^{\prime }\in \mathfrak{C}\left( {Q}\right) .
\end{equation*}%
Thus 
\begin{equation*}
\dim L_{Q;k}^{2}\left( \mu \right) \leq \#\{Q^{\prime }\in \mathfrak{C}{Q}%
:\left\vert Q^{\prime }\right\vert _{\mu }>0\}B-B\leq (2^{n}-1)B\leq
(2^{n}-1)A_{k,n}.
\end{equation*}%
Finally it is well-known that the number of nonnegative solutions to $\alpha
_{1}+\cdots +\alpha _{n}=j$ is $\binom{n-1+j}{j}$ (simply choose $n-1$ boxes
from a row of $n+j-1$ boxes and put a single ball in each of the unchosen
boxes - then let $\alpha _{i}$ be the number of balls between the $\left(
i-1\right) ^{th}$ box and the $i^{th}$ box), and so $A_{k,n}=\sum_{j=0}^{k-1}%
\binom{n-1+j}{j}=\binom{n-1+k}{k-1}$ by induction on $k$.
\end{proof}

Now we restrict attention to a fixed dyadic grid $\mathcal{D}$ in $\mathbb{R}%
^{n}$. For $P,Q\in \mathcal{D}$ dyadic cubes, the subspaces $%
L_{P;k}^{2}\left( \mu \right) $ and $L_{Q;k}^{2}\left( \mu \right) $ are
orthogonal for $P\neq Q$, i.e. $\left\langle f,g\right\rangle _{L^{2}\left(
\mu \right) }=0$ for $f\in L_{P;k}^{2}\left( \mu \right) $ and $g\in
L_{Q;k}^{2}\left( \mu \right) $. Indeed, the only case that needs checking
is when either $P\subsetneqq Q$ or $Q\subsetneqq P$. If $P\subsetneq Q$,
then the restriction of $g$ to $P$ is a polynomial of degree less than or
equal to $k-1$. But by definition, $L_{P;k}^{2}(\mu )$ is orthogonal to such
functions and so $\left\langle f,g\right\rangle _{L^{2}\left( \mu \right)
}=0 $. Of course similar reasoning holds in the case $Q\subsetneq P$.

Now define%
\begin{eqnarray*}
&&\mathcal{F}_{\infty }^{k}\left( \mu \right) \equiv \left\{ \alpha \in 
\mathbb{Z}_{+}^{n}:\left\vert \alpha \right\vert \leq k-1:x^{\alpha }\in
L^{2}\left( \mu \right) \right\} \ , \\
&&\ \ \ \ \ \ \ \ \ \ \ \ \ \ \ \text{and }\mathcal{P}_{\mathbb{R}%
^{n}}^{k}\left( \mu \right) \equiv \limfunc{Span}\left\{ x^{\alpha }\right\}
_{\alpha \in \mathcal{F}_{\infty }^{k}}\ .
\end{eqnarray*}%
We claim that $L^{2}\left( \mu \right) $ is the Hilbert space direct sum of $%
\mathcal{P}_{\mathbb{R}^{n}}^{k}\left( \mu \right) $ in $L^{2}\left( \mu
\right) $ and the finite dimensional subspaces $\left\{ L_{Q;k}^{2}\left(
\mu \right) \right\} _{Q\in \mathcal{D}}$, namely 
\begin{equation*}
L^{2}\left( \mu \right) =\mathcal{P}_{\mathbb{R}^{n}}^{k}\left( \mu \right)
\oplus \left( {\Large \oplus }_{Q\in \mathcal{D}}L_{Q;k}^{2}\left( \mu
\right) \right) .
\end{equation*}%
To see this, fix a large dyadic cube $P\in \mathcal{D}$. Set%
\begin{equation*}
\mathcal{P}_{P}^{k}\equiv \left\{ f=\mathbf{1}_{P}\left( x\right) \
p_{k}\left( x\right) :\alpha \in \mathbb{Z}_{+}^{n}:\left\vert \alpha
\right\vert \leq k-1\right\}
\end{equation*}%
to be the linear space of restrictions to $P$ of polynomials $p_{k}\left(
x\right) $ having degree at most $k-1$. Then for $Q\subset P$, we have%
\begin{eqnarray*}
L_{Q;1}^{2}\left( \mu \right) &=&\left\{ f=\dsum\limits_{Q^{\prime }\in 
\mathfrak{C}\left( Q\right) }a_{Q^{\prime }}\mathbf{1}_{Q^{\prime
}}:a_{Q^{\prime }}\in \mathbb{R},\int_{Q}fd\mu =0\right\} \\
&\subset &\overline{\limfunc{Span}}\left\{ \mathcal{P}_{P}^{k},\left\{
L_{R;k}^{2}\left( \mu \right) \right\} _{R\in \mathcal{D}:\ Q\subset
R\subset P}\right\} .
\end{eqnarray*}%
Now let $P$ tend to infinity to conclude that%
\begin{equation*}
L_{Q;1}^{2}\left( \mu \right) \subset \overline{\limfunc{Span}}\left\{ 
\mathcal{P}_{\mathbb{R}^{n}}^{k}\left( \mu \right) ,\left\{
L_{R;k}^{2}\left( \mu \right) \right\} _{R\in \mathcal{D}:\ Q\subset
R}\right\} .
\end{equation*}%
But we already know that the Haar spaces $\mathcal{P}_{\mathbb{R}%
^{n}}^{1}\left( \mu \right) $ and $L_{Q;1}^{2}\left( \mu \right) $ form a
direct sum decomposition of $L^{2}\left( \mu \right) $, i.e. 
\begin{equation*}
L^{2}\left( \mu \right) =\mathcal{P}_{\mathbb{R}^{n}}^{1}\left( \mu \right)
\oplus \left( {\Large \oplus }_{Q\in \mathcal{D}}L_{Q;1}^{2}\left( \mu
\right) \right) ,
\end{equation*}%
and hence we see that 
\begin{equation*}
L^{2}\left( \mu \right) =\mathcal{P}_{\mathbb{R}^{n}}^{1}\left( \mu \right)
\oplus \left( {\Large \oplus }_{Q\in \mathcal{D}}L_{Q;1}^{2}\left( \mu
\right) \right) \subset \mathcal{P}_{\mathbb{R}^{n}}^{k}\left( \mu \right)
\oplus \left( {\Large \oplus }_{Q\in \mathcal{D}}L_{Q;k}^{2}\left( \mu
\right) \right) \subset L^{2}\left( \mu \right) .
\end{equation*}%
Formula (\ref{also written}) gives the telescoping identities (\ref%
{telescoping}), and the moment conditions (\ref{mom con}) are immediate from
the definition of $L_{Q;k}^{2}\left( \mu \right) $. This completes the proof
of Theorem \ref{main1}.

\subsection{Uniqueness and degeneracy}

In this subsection we begin to investigate the lack of uniqueness of the
orthogonal projections $\left\{ \bigtriangleup _{\mathbb{R}^{n};k}^{\mu
}\right\} \cup \left\{ \bigtriangleup _{Q;k}^{\mu }\right\} _{Q\in \mathcal{D%
}^{n}}$ and their degeneracy when%
\begin{equation*}
d_{Q;k}<\left( 2^{n}-1\right) \binom{n+k-1}{n}
\end{equation*}%
is not maximal. We can of course use the Gram-Schmidt orthogonalization
algorithm to find an orthonormal basis $\left\{ k_{\mathbb{R}^{n}}^{\mu
,a}\right\} _{a\in \Gamma _{n}^{k}\left( \mathbb{R}^{n}\right) }$ of $%
\mathcal{P}_{\mathbb{R}^{n}}^{k}\left( \mu \right) $ where $\Gamma
_{n}^{k}\left( \mathbb{R}^{n}\right) $ is any convenient index set with
cardinality equal to $\dim \mathcal{P}_{\mathbb{R}^{n}}^{k}\left( \mu
\right) =\#\mathcal{F}_{\infty }^{k}\left( \mu \right) $, and also an
orthonormal basis $\left\{ h_{Q}^{\mu ,a}\right\} _{a\in \Gamma _{n}^{k}}$
of $L_{Q;k}^{2}\left( \mu \right) $ where $\Gamma _{n}^{k}$ is any
convenient index set with cardinality equal to $d_{Q;k}$, the dimension of $%
L_{Q}^{2,1}\left( \mu \right) $, i.e. 
\begin{equation*}
\#\Gamma _{n}^{k}\left( Q\right) =\dim L_{Q;k}^{2}\left( \mu \right) .
\end{equation*}%
Then for any $J\in \mathcal{D}\cup \left\{ \mathbb{R}^{n}\right\} $, the set 
$\left\{ k_{J}^{\mu ,a}\right\} _{a\in \Gamma _{n}^{k}\left( J\right) }\cup
\left\{ h_{Q}^{\mu ,a}\right\} _{a\in \Gamma _{n}^{k}\left( Q\right) \text{
and }Q\in \mathcal{D}\text{ with }Q\subset J}$ $\ $is an orthonormal basis
for $L^{2}\left( \mathbf{1}_{J}\mu \right) $.

In the case of dimension $n=1$ with $k=1$ vanishing moment, Proposition \ref%
{number} shows that the dimension $d_{Q;1}$ of the subspace $%
L_{Q;1}^{2}\left( \mu \right) $ satisfies $0\leq d_{Q;1}\leq \left(
2^{1}-1\right) \left( 
\begin{array}{c}
1+1-1 \\ 
1-1%
\end{array}%
\right) =1$. So if $L_{Q;1}^{2}\left( \mu \right) \neq \left\{ 0\right\} $,
the orthogonal projection $\bigtriangleup _{Q;1}^{\mu }$ is one-dimensional,
and hence is given by $\bigtriangleup _{Q;1}^{\mu }f=\frac{\left\langle
f,h_{Q}^{\mu }\right\rangle }{\left\langle h_{Q}^{\mu },h_{Q}^{\mu
}\right\rangle }h_{Q}^{\mu }$ for a function $h_{Q}^{\mu }$, which has the
especially simple formula given by (\ref{simple formula}). We next
investigate to what extent one can find "nice" explicit bases of $%
L_{Q;k}^{2}\left( \mu \right) $ in dimension $n=1$, where in this case $%
0\leq d_{Q;k}\leq \left( 2^{1}-1\right) \left( 
\begin{array}{c}
1+k-1 \\ 
k-1%
\end{array}%
\right) =k$.

\subsection{Proof of Theorem \protect\ref{dim one}}

To prepare for the proof of Theorem \ref{dim one}, we first turn to the\
explicit construction of weighted Alpert bases in dimension $n=1$ when the
number of vanishing moments is $k\geq 2$.

\subsubsection{An explicit basis for $n=1$ and $k\geq 2$}

Let $\mu $ be a locally finite positive Borel measure on $\mathbb{R}$ that
satisfies the Alpert nondegeneracy condition%
\begin{equation}
\boldsymbol{M}_{J,k}\equiv \left[ 
\begin{array}{ccccc}
\left\vert J^{\left( 0\right) }\right\vert _{\mu } & \left\vert J^{\left(
1\right) }\right\vert _{\mu } & \cdots & \left\vert J^{\left( k-2\right)
}\right\vert _{\mu } & \left\vert J^{\left( k-1\right) }\right\vert _{\mu }
\\ 
\left\vert J^{\left( 1\right) }\right\vert _{\mu } & \left\vert J^{\left(
2\right) }\right\vert _{\mu } & \ddots & \left\vert J^{\left( k-1\right)
}\right\vert _{\mu } & \left\vert J^{\left( k\right) }\right\vert _{\mu } \\ 
\vdots & \ddots & \ddots & \ddots & \vdots \\ 
\left\vert J^{\left( k-2\right) }\right\vert _{\mu } & \left\vert J^{\left(
k-1\right) }\right\vert _{\mu } & \ddots & \left\vert J^{\left( 2k-4\right)
}\right\vert _{\mu } & \left\vert J^{\left( 2k-3\right) }\right\vert _{\mu }
\\ 
\left\vert J^{\left( k-1\right) }\right\vert _{\mu } & \left\vert J^{\left(
k\right) }\right\vert _{\mu } & \cdots & \left\vert J^{\left( 2k-3\right)
}\right\vert _{\mu } & \left\vert J^{\left( 2k-2\right) }\right\vert _{\mu }%
\end{array}%
\right] \succ 0,  \label{Alpert nondeg k}
\end{equation}%
for all $J\in \mathcal{D}$, and where we denote a positive definite matrix $%
A $ by $A\succ 0$. For each $x\in J$ we have that%
\begin{equation*}
\boldsymbol{M}\left( x\right) \equiv \left[ 
\begin{array}{ccccc}
1 & x & \cdots & x^{k-2} & x^{k-1} \\ 
x & x^{2} & \ddots & x^{k-3} & x^{k-2} \\ 
\vdots & \ddots & \ddots & \ddots & \vdots \\ 
x^{k-2} & x^{k-3} & \ddots & x^{2k-4} & x^{2k-3} \\ 
x^{k-1} & x^{k-2} & \cdots & x^{2k-3} & x^{2k-2}%
\end{array}%
\right] =\left[ 
\begin{array}{c}
1 \\ 
x \\ 
\vdots \\ 
x^{k-2} \\ 
x^{k-1}%
\end{array}%
\right] \left[ 
\begin{array}{ccccc}
1 & x & \cdots & x^{k-2} & x^{k-1}%
\end{array}%
\right]
\end{equation*}%
is a dyad, namely a rank one nonnegative semidefinite matrix. Thus%
\begin{equation*}
\boldsymbol{M}_{J,k}=\int_{J}\left[ 
\begin{array}{ccccc}
1 & x & \cdots & x^{k-2} & x^{k-1} \\ 
x & x^{2} & \ddots & x^{k-3} & x^{k-2} \\ 
\vdots & \ddots & \ddots & \ddots & \vdots \\ 
x^{k-2} & x^{k-3} & \ddots & x^{2k-4} & x^{2k-3} \\ 
x^{k-1} & x^{k-2} & \cdots & x^{2k-3} & x^{2k-2}%
\end{array}%
\right] d\mu \left( x\right) =\int_{J}\boldsymbol{M}\left( x\right) d\mu
\left( x\right)
\end{equation*}%
is also a nonnegative semidefinite matrix.

We now claim that the matrix $\boldsymbol{M}_{J,k}$ is positive definite 
\emph{if and only if} the functions%
\begin{equation*}
\left\{ \mathbf{1}_{J}\left( x\right) ,x\mathbf{1}_{J}\left( x\right) ,\
...,\ x^{k-1}\mathbf{1}_{J}\left( x\right) \right\}
\end{equation*}%
are linearly independent in $L^{2}\left( \mu \right) $. Indeed, in the
special case $k=2$ we have%
\begin{equation*}
\det \boldsymbol{M}_{J,2}=\det \left[ 
\begin{array}{cc}
\left\vert J^{\left( 0\right) }\right\vert _{\mu } & \left\vert J^{\left(
1\right) }\right\vert _{\mu } \\ 
\left\vert J^{\left( 1\right) }\right\vert _{\mu } & \left\vert J^{\left(
2\right) }\right\vert _{\mu }%
\end{array}%
\right] >0
\end{equation*}%
if and only if%
\begin{equation*}
\left\vert J^{\left( 1\right) }\right\vert _{\mu }^{2}=\left( \int_{J}xd\mu
\left( x\right) \right) ^{2}\leq \left( \int_{J}1^{2}d\mu \left( x\right)
\right) \left( \int_{J}x^{2}d\mu \left( x\right) \right) =\left\vert
J^{\left( 0\right) }\right\vert _{\mu }\left\vert J^{\left( 2\right)
}\right\vert _{\mu }\ ,
\end{equation*}%
with strict inequality if and only if the functions $\mathbf{1}_{J}\left(
x\right) $ and $x\mathbf{1}_{J}\left( x\right) $ are linearly independent in 
$L^{2}\left( \mu \right) $. Here is the general case.

\begin{theorem}
Let $\mu $ be a locally finite positive Borel measure on $\mathbb{R}$. Then:

\begin{enumerate}
\item the $k\times k$ matrix of moments%
\begin{equation*}
\boldsymbol{M}_{J,k}=\int \left[ 
\begin{array}{ccccc}
1 & x & \cdots & x^{k-2} & x^{k-1} \\ 
x & x^{2} & \ddots & x^{k-3} & x^{k-2} \\ 
\vdots & \ddots & \ddots & \ddots & \vdots \\ 
x^{k-2} & x^{k-3} & \ddots & x^{2k-4} & x^{2k-3} \\ 
x^{k-1} & x^{k-2} & \cdots & x^{2k-3} & x^{2k-2}%
\end{array}%
\right] d\mu \left( x\right)
\end{equation*}%
is nonnegative semidefinite, and

\item $\boldsymbol{M}_{J,k}$ has rank $\ell $ if and only if the span of the
functions $1,x,x^{2},..,x^{k-1}$ has dimension $\ell $ in $L^{2}\left( \mu
\right) $, i.e.%
\begin{equation*}
\limfunc{rank}\boldsymbol{M}_{J,k}=\dim \limfunc{Span}\left\{ x^{j}\right\}
_{j=0}^{k-1}\ .
\end{equation*}
\end{enumerate}
\end{theorem}

\begin{proof}
With $V_{k}\left( x\right) ^{\limfunc{tr}}\equiv \left[ 
\begin{array}{ccccc}
1 & x & \cdots & x^{k-2} & x^{k-1}%
\end{array}%
\right] $, we compute%
\begin{equation*}
\left[ 
\begin{array}{ccccc}
1 & x & \cdots & x^{k-2} & x^{k-1} \\ 
x & x^{2} & \ddots & x^{k-3} & x^{k-2} \\ 
\vdots & \ddots & \ddots & \ddots & \vdots \\ 
x^{k-2} & x^{k-3} & \ddots & x^{2k-4} & x^{2k-3} \\ 
x^{k-1} & x^{k-2} & \cdots & x^{2k-3} & x^{2k-2}%
\end{array}%
\right] =\left[ 
\begin{array}{c}
1 \\ 
x \\ 
\vdots \\ 
x^{k-2} \\ 
x^{k-1}%
\end{array}%
\right] \left[ 
\begin{array}{ccccc}
1 & x & \cdots & x^{k-2} & x^{k-1}%
\end{array}%
\right] =V_{k}\left( x\right) V_{k}\left( x\right) ^{\limfunc{tr}},
\end{equation*}%
and so%
\begin{equation*}
\boldsymbol{M}_{J,k}\equiv \int_{\mathbb{R}^{k}}V_{k}\left( x\right)
V_{k}\left( x\right) ^{\limfunc{tr}}d\mu \left( x\right)
\end{equation*}%
is a positive integral of dyads, hence nonnegative semidefinite. Now note
that the quadratic form%
\begin{equation*}
\xi ^{\limfunc{tr}}\boldsymbol{M}_{J,k}\xi =\int_{\mathbb{R}^{k}}\left( \xi
^{\limfunc{tr}}V_{k}\left( x\right) \right) \left( V_{k}\left( x\right) ^{%
\limfunc{tr}}\xi \right) d\mu \left( x\right) =\int_{\mathbb{R}^{k}}\left(
\xi \cdot V_{k}\left( x\right) \right) ^{2}d\mu \left( x\right)
\end{equation*}%
is strictly positive if and only if $\limfunc{Span}\left\{ V_{k}\left(
x\right) :x\in \limfunc{supp}\ \mu \right\} $ has dimension $k$. Indeed, $%
\dim \limfunc{Span}\left\{ V_{k}\left( x\right) :x\in \limfunc{supp}\ \mu
\right\} =k$ if and only if the function $\xi \cdot V_{k}\left( x\right) $
is not trivial for every $\xi \in \mathbb{R}^{k}\setminus \left\{ 0\right\} $%
, and this in turn happens if and only if $\xi ^{\limfunc{tr}}\boldsymbol{M}%
_{J,k}\xi >0$ for all $\xi \in \mathbb{R}^{k}\setminus \left\{ 0\right\} $.
Similarly we have $\dim \limfunc{Span}\left\{ V_{k}\left( x\right) :x\in 
\limfunc{supp}\ \mu \right\} =\ell $ if and only if the vector subspace of $%
\mathbb{R}^{k}$ defined by 
\begin{equation*}
S\equiv \left\{ \xi \in \mathbb{R}^{k}:\xi \cdot V_{k}\left( x\right) \text{
is not trivial}\right\}
\end{equation*}%
has dimension $\ell $, and this in turn happens if and only if $\dim 
\limfunc{Span}\left\{ x^{j}\right\} _{j=0}^{k-1}=\ell $.
\end{proof}

\begin{remark}
The van der Monde determinant formula shows that $\dim \limfunc{Span}\left\{
x^{j}\right\} _{j=0}^{k-1}=\ell <k$ if and only if the measure $\mu $ is a
sum of positive multiples of $\ell $ point masses, i.e. $\mu
=\sum_{j=1}^{\ell }c_{j}\delta _{x_{j}}$ with $c_{j}>0$ and the $x_{j}$
distinct.
\end{remark}

\subsubsection{An orthonormal basis of the multiresolution projections}

Let $\mu $ be a locally finite positive Borel measure on $\mathbb{R}$ that
satisfies the Alpert nondegeneracy condition (\ref{Alpert nondeg k}), i.e.
the determinant of all principal $\ell \times \ell $ submatrices $%
\boldsymbol{M}_{J,\mathbf{i}}$ is positive for $\mathbf{i}=\left(
i_{1},i_{2},...,i_{\ell }\right) $ and $1\leq \ell \leq k$, 
\begin{eqnarray*}
&&\det \boldsymbol{M}_{J,\mathbf{i}}>0,\ \ \ \ \ \text{for all }1\leq \ell
\leq k\text{, and all }J=I_{j}^{m},\ m\in \mathbb{Z}_{+},\ 0\leq j\leq
2^{m}-1; \\
&&\text{where }\boldsymbol{M}_{J,\mathbf{i}}\equiv \left[ 
\begin{array}{ccccc}
\left\vert J^{\left( 2i_{1}\right) }\right\vert _{\mu } & \left\vert
J^{\left( i_{1}+i_{2}\right) }\right\vert _{\mu } & \cdots & \left\vert
J^{\left( i_{1}+i_{\ell -1}\right) }\right\vert _{\mu } & \left\vert
J^{\left( i_{1}+i_{\ell }\right) }\right\vert _{\mu } \\ 
\left\vert J^{\left( i_{2}+i_{1}\right) }\right\vert _{\mu } & \left\vert
J^{\left( 2i_{2}\right) }\right\vert _{\mu } & \ddots & \left\vert J^{\left(
i_{2}+i_{\ell -1}\right) }\right\vert _{\mu } & \left\vert J^{\left(
i_{2}+i_{\ell }\right) }\right\vert _{\mu } \\ 
\vdots & \ddots & \ddots & \ddots & \vdots \\ 
\left\vert J^{\left( i_{\ell -1}+i_{1}\right) }\right\vert _{\mu } & 
\left\vert J^{\left( i_{\ell -1}+i_{2}\right) }\right\vert _{\mu } & \ddots
& \left\vert J^{\left( 2i_{\ell -1}\right) }\right\vert _{\mu } & \left\vert
J^{\left( i_{\ell -1}+i_{\ell }\right) }\right\vert _{\mu } \\ 
\left\vert J^{\left( i_{\ell }+i_{1}\right) }\right\vert _{\mu } & 
\left\vert J^{\left( i_{\ell }+i_{2}\right) }\right\vert _{\mu } & \cdots & 
\left\vert J^{\left( i_{\ell }+i_{\ell -1}\right) }\right\vert _{\mu } & 
\left\vert J^{\left( 2i_{\ell }\right) }\right\vert _{\mu }%
\end{array}%
\right] .
\end{eqnarray*}%
If for $\mathbf{i}=\left( i_{1},i_{2},...,i_{\ell }\right) $ we set $V_{%
\mathbf{i}}\left( x\right) ^{\limfunc{tr}}\equiv \left[ 
\begin{array}{ccccc}
x^{i_{1}} & x^{i_{2}} & \cdots & x^{i_{\ell -1}} & x^{i_{\ell }}%
\end{array}%
\right] $, then by the above theorem, 
\begin{equation*}
\boldsymbol{M}_{J,\mathbf{i}}\equiv \int_{J}V_{\mathbf{i}}\left( x\right) V_{%
\mathbf{i}}\left( x\right) ^{\limfunc{tr}}d\mu \left( x\right)
\end{equation*}%
is nonnegative semidefinite; and is positive definite if and only if the the
span of the functions $x^{i_{1}},x^{i_{2}},...,x^{i_{\ell }}$ has dimension $%
\ell $ in $L^{2}\left( \mu \right) $, i.e.%
\begin{equation*}
\limfunc{rank}\boldsymbol{M}_{J,\mathbf{i}}=\dim \limfunc{Span}\left\{
x^{i_{j}}\right\} _{j=1}^{\ell }\ .
\end{equation*}

\begin{remark}
Recall that a matrix $M\,$is nonnegative semidefinite if and only if $\det 
\boldsymbol{M}_{\mathbf{i}}\geq 0$ for all principal submatrices $%
\boldsymbol{M}_{\mathbf{i}}$ (the numbers $\det \boldsymbol{M}_{\mathbf{i}}$
are usually referred to as principal minors); and that a matrix $M\,$is
positive definite if and only if $\det \boldsymbol{M}_{\mathbf{i}}>0$ for
the special subset of leading principal submatrices $\boldsymbol{M}_{\mathbf{%
i}}$ having $\mathbf{i}=\left( 1,2,...,\ell \right) $ and $1\leq \ell \leq k$
(in which case $\det \boldsymbol{M}_{\mathbf{i}}>0$ for all principal
submatrices $\boldsymbol{M}_{\mathbf{i}}$).
\end{remark}

Recall that we can use the Gram-Schmidt orthogonalization algorithm to find
an orthonormal basis $\left\{ h_{Q}^{\mu ,a}\right\} _{a\in \Gamma _{1}^{k}}$
of $L_{Q;k}^{2}\left( \mu \right) $ where $\Gamma _{1}^{k}$ is any
convenient index set with cardinality equal to $d_{Q;k}$, the dimension of $%
L_{Q}^{2,1}\left( \mu \right) $, i.e. 
\begin{equation*}
\#\Gamma _{1}^{k}=\dim L_{Q;k}^{2}\left( \mu \right) .
\end{equation*}%
Then $\left\{ x^{\alpha }\right\} _{\alpha \in \mathcal{F}_{\infty }}\cup
\left\{ h_{Q}^{\mu ,a}\right\} _{a\in \Gamma _{n}^{k}\text{ and }Q\in 
\mathcal{D}}$ is an orthonormal basis for $L_{\mathbb{R}^{n}}^{2}\left( \mu
\right) $

In the case of dimension $n=1$ with $k\geq 2$ vanishing moments, the above
proposition shows that the dimension $d_{Q;k}$ of the subspace $%
L_{Q;k}^{2}\left( \mu \right) $ satisfies $0\leq d_{Q;k}\leq \left(
2^{1}-1\right) \left( 
\begin{array}{c}
1+k-1 \\ 
k-1%
\end{array}%
\right) =k$, and in the case where $\boldsymbol{M}_{J,k}$ is positive
definite, we have $d_{Q;k}=k$. Now we begin an explicit construction of an
`Alpert' orthonormal basis.

We require the Alpert functions%
\begin{eqnarray*}
a_{I}^{\mu ,\ell }\left( x\right) &\equiv &\left( \alpha _{k-1}^{\ell
}x^{k-1}+\alpha _{k-2}^{\ell }x^{k-2}+...+\alpha _{1}^{\ell }x+\alpha
_{0}^{\ell }\right) \mathbf{1}_{I_{\limfunc{left}}}\left( x\right) \\
&&+\left( \beta _{k-1}^{\ell }x^{k-1}+\beta _{k-2}^{\ell }x^{k-2}+...+\beta
_{1}^{\ell }x+\beta _{0}^{\ell }\right) \mathbf{1}_{I_{\limfunc{right}%
}}\left( x\right) ,
\end{eqnarray*}%
namely%
\begin{eqnarray*}
a_{I}^{\mu ,1}\left( x\right) &\equiv &\left( \alpha
_{k-1}^{1}x^{k-1}+\alpha _{k-2}^{1}x^{k-2}+...+\alpha _{1}^{1}x+\alpha
_{0}^{1}\right) \mathbf{1}_{I_{\limfunc{left}}}\left( x\right) \\
&&+\left( \beta _{k-1}^{1}x^{k-1}+\beta _{k-2}^{1}x^{k-2}+...+\beta
_{1}^{1}x+\beta _{0}^{1}\right) \mathbf{1}_{I_{\limfunc{right}}}\left(
x\right) , \\
a_{I}^{\mu ,2}\left( x\right) &\equiv &\left( \alpha
_{k-1}^{2}x^{k-1}+\alpha _{k-2}^{2}x^{k-2}+...+\alpha _{1}^{2}x+\alpha
_{0}^{2}\right) \mathbf{1}_{I_{\limfunc{left}}}\left( x\right) \\
&&+\left( \beta _{k-1}^{2}x^{k-1}+\beta _{k-2}^{2}x^{k-2}+...+\beta
_{1}^{2}x+\beta _{0}^{2}\right) \mathbf{1}_{I_{\limfunc{right}}}\left(
x\right) , \\
&&\vdots \\
a_{I}^{\mu ,k}\left( x\right) &\equiv &\left( \alpha
_{k-1}^{k}x^{k-1}+\alpha _{k-2}^{k}x^{k-2}+...+\alpha _{1}^{k}x+\alpha
_{0}^{k}\right) \mathbf{1}_{I_{\limfunc{left}}}\left( x\right) \\
&&+\left( \beta _{k-1}^{k}x^{k-1}+\beta _{k-2}^{k}x^{k-2}+...+\beta
_{1}^{k}x+\beta _{0}^{k}\right) \mathbf{1}_{I_{\limfunc{right}}}\left(
x\right) ,
\end{eqnarray*}%
to satisfy the $k^{2}$ moment properties,%
\begin{equation*}
\int a_{I}^{\mu ,\ell }\left( x\right) x^{i}d\mu \left( x\right) =0,\ \ \ \
\ \text{for all }1\leq \ell \leq k\text{ and }0\leq i\leq k+\ell -2,
\end{equation*}%
the $\left( 
\begin{array}{c}
k \\ 
2%
\end{array}%
\right) $ orthogonality properties,%
\begin{equation*}
\int a_{I}^{\mu ,\ell }\left( x\right) a_{I}^{\mu ,\ell ^{\prime }}\left(
x\right) d\mu \left( x\right) =0,\ \ \ \ \ \text{for all }1\leq \ell ,\ell
^{\prime }\leq k,
\end{equation*}%
and the $k$ normalization properties, 
\begin{equation*}
\left( \int \left\vert a_{I}^{\mu ,\ell }\left( x\right) \right\vert
^{2}d\mu \left( x\right) \right) ^{\frac{1}{2}}=1,\ \ \ \ \ \text{for all }%
1\leq \ell \leq k.
\end{equation*}

\begin{description}
\item[Dimension count] Note that $k^{2}+\left( 
\begin{array}{c}
k \\ 
2%
\end{array}%
\right) +k=2k^{2}-\left( 
\begin{array}{c}
k \\ 
2%
\end{array}%
\right) $, so that there are $\left( 
\begin{array}{c}
k \\ 
2%
\end{array}%
\right) $ degrees of freedom remaining in the choice of the $2k^{2}$
coefficients $\left\{ \alpha _{i}^{\ell },\beta _{i}^{\ell }\right\} _{1\leq
\ell \leq k\text{ and }0\leq i\leq k-1}$. We will later show that we can
impose $\left( 
\begin{array}{c}
k \\ 
2%
\end{array}%
\right) $ \emph{additional} moment conditions.
\end{description}

So with the definition $\left\vert J^{\left( i\right) }\right\vert _{\mu
}\equiv \int_{J}x^{i}d\mu \left( x\right) $, we first tackle the moment
properties:%
\begin{eqnarray*}
0 &=&\int a_{I}^{\mu ,\ell }\left( x\right) x^{i}d\mu \left( x\right)
=\int_{I_{\limfunc{left}}}\left( \alpha _{k-1}^{\ell }x^{k-1}+\alpha
_{k-2}^{\ell }x^{k-2}+...+\alpha _{1}^{\ell }x+\alpha _{0}^{\ell }\right)
x^{i}d\mu \left( x\right) \\
&&+\int_{I_{\limfunc{right}}}\left( \beta _{k-1}^{\ell }x^{k-1}+\beta
_{k-2}^{\ell }x^{k-2}+...+\beta _{1}^{\ell }x+\beta _{0}^{\ell }\right)
x^{i}d\mu \left( x\right) \\
&=&\alpha _{k-1}^{\ell }\left\vert I_{\limfunc{left}}^{\left( i+k-1\right)
}\right\vert _{\mu }+\alpha _{k-2}^{\ell }\left\vert I_{\limfunc{left}%
}^{\left( i+k-2\right) }\right\vert _{\mu }+...+\alpha _{0}^{\ell
}\left\vert I_{\limfunc{left}}^{\left( i\right) }\right\vert _{\mu } \\
&&+\beta _{k-1}^{\ell }\left\vert I_{\limfunc{right}}^{\left( i+k-1\right)
}\right\vert _{\mu }+\beta _{k-2}^{\ell }\left\vert I_{\limfunc{right}%
}^{\left( i+k-2\right) }\right\vert _{\mu }+...+\beta _{0}^{\ell }\left\vert
I_{\limfunc{right}}^{\left( i\right) }\right\vert _{\mu },
\end{eqnarray*}%
which lead to%
\begin{equation*}
\left[ 
\begin{array}{ccc}
\left\vert I_{\limfunc{left}}^{\left( i\right) }\right\vert _{\mu } & \cdots
& \left\vert I_{\limfunc{left}}^{\left( i+k-1\right) }\right\vert _{\mu }%
\end{array}%
\right] \left( 
\begin{array}{c}
\alpha _{0}^{\ell } \\ 
\vdots \\ 
\alpha _{k-1}^{\ell }%
\end{array}%
\right) +\left[ 
\begin{array}{ccc}
\left\vert I_{\limfunc{right}}^{\left( i\right) }\right\vert _{\mu } & \cdots
& \left\vert I_{\limfunc{right}}^{\left( i+k-1\right) }\right\vert _{\mu }%
\end{array}%
\right] \left( 
\begin{array}{c}
\beta _{0}^{\ell } \\ 
\vdots \\ 
\beta _{k-1}^{\ell }%
\end{array}%
\right) =0,\ \ \ \ \ 0\leq i\leq k-1,
\end{equation*}%
which in matrix form is,%
\begin{eqnarray*}
\left( 
\begin{array}{c}
0 \\ 
0 \\ 
\vdots \\ 
0%
\end{array}%
\right) &=&\left[ 
\begin{array}{ccccc}
\left\vert I_{\limfunc{left}}^{\left( 0\right) }\right\vert _{\mu } & 
\left\vert I_{\limfunc{left}}^{\left( 1\right) }\right\vert _{\mu } & \cdots
& \left\vert I_{\limfunc{left}}^{\left( k-2\right) }\right\vert _{\mu } & 
\left\vert I_{\limfunc{left}}^{\left( k-1\right) }\right\vert _{\mu } \\ 
\left\vert I_{\limfunc{left}}^{\left( 1\right) }\right\vert _{\mu } & 
\left\vert I_{\limfunc{left}}^{\left( 2\right) }\right\vert _{\mu } & \cdots
& \left\vert I_{\limfunc{left}}^{\left( k-1\right) }\right\vert _{\mu } & 
\left\vert I_{\limfunc{left}}^{\left( k\right) }\right\vert _{\mu } \\ 
\vdots & \vdots &  &  & \vdots \\ 
\left\vert I_{\limfunc{left}}^{\left( k-2\right) }\right\vert _{\mu } & 
\left\vert I_{\limfunc{left}}^{\left( k-1\right) }\right\vert _{\mu } & 
\cdots & \left\vert I_{\limfunc{left}}^{\left( 2k-4\right) }\right\vert
_{\mu } & \left\vert I_{\limfunc{left}}^{\left( 2k-3\right) }\right\vert
_{\mu } \\ 
\left\vert I_{\limfunc{left}}^{\left( k-1\right) }\right\vert _{\mu } & 
\left\vert I_{\limfunc{left}}^{\left( k\right) }\right\vert _{\mu } & \cdots
& \left\vert I_{\limfunc{left}}^{\left( 2k-3\right) }\right\vert _{\mu } & 
\left\vert I_{\limfunc{left}}^{\left( 2k-2\right) }\right\vert _{\mu }%
\end{array}%
\right] \left( 
\begin{array}{c}
\alpha _{0}^{\ell } \\ 
\alpha _{1}^{\ell } \\ 
\vdots \\ 
\alpha _{k-2}^{\ell } \\ 
\alpha _{k-1}^{\ell }%
\end{array}%
\right) \\
&&+\left[ 
\begin{array}{ccccc}
\left\vert I_{\limfunc{right}}^{\left( 0\right) }\right\vert _{\mu } & 
\left\vert I_{\limfunc{right}}^{\left( 1\right) }\right\vert _{\mu } & \cdots
& \left\vert I_{\limfunc{right}}^{\left( k-2\right) }\right\vert _{\mu } & 
\left\vert I_{\limfunc{right}}^{\left( k-1\right) }\right\vert _{\mu } \\ 
\left\vert I_{\limfunc{right}}^{\left( 1\right) }\right\vert _{\mu } & 
\left\vert I_{\limfunc{right}}^{\left( 2\right) }\right\vert _{\mu } & \cdots
& \left\vert I_{\limfunc{right}}^{\left( k-1\right) }\right\vert _{\mu } & 
\left\vert I_{\limfunc{right}}^{\left( k\right) }\right\vert _{\mu } \\ 
\vdots & \vdots &  &  & \vdots \\ 
\left\vert I_{\limfunc{right}}^{\left( k-2\right) }\right\vert _{\mu } & 
\left\vert I_{\limfunc{right}}^{\left( k-1\right) }\right\vert _{\mu } & 
\cdots & \left\vert I_{\limfunc{right}}^{\left( 2k-4\right) }\right\vert
_{\mu } & \left\vert I_{\limfunc{right}}^{\left( 2k-3\right) }\right\vert
_{\mu } \\ 
\left\vert I_{\limfunc{right}}^{\left( k-1\right) }\right\vert _{\mu } & 
\left\vert I_{\limfunc{right}}^{\left( k\right) }\right\vert _{\mu } & \cdots
& \left\vert I_{\limfunc{right}}^{\left( 2k-3\right) }\right\vert _{\mu } & 
\left\vert I_{\limfunc{right}}^{\left( 2k-2\right) }\right\vert _{\mu }%
\end{array}%
\right] \left( 
\begin{array}{c}
\beta _{0}^{\ell } \\ 
\beta _{1}^{\ell } \\ 
\vdots \\ 
\beta _{k-2}^{\ell } \\ 
\beta _{k-1}^{\ell }%
\end{array}%
\right) .
\end{eqnarray*}%
But this says that%
\begin{equation}
\left[ 
\begin{array}{ccc}
\left\vert I_{\limfunc{left}}^{\left( 0\right) }\right\vert _{\mu } & \cdots
& \left\vert I_{\limfunc{left}}^{\left( k-1\right) }\right\vert _{\mu } \\ 
\vdots &  & \vdots \\ 
\left\vert I_{\limfunc{left}}^{\left( k-1\right) }\right\vert _{\mu } & 
\cdots & \left\vert I_{\limfunc{left}}^{\left( 2k-2\right) }\right\vert
_{\mu }%
\end{array}%
\right] \left( 
\begin{array}{c}
\alpha _{0}^{\ell } \\ 
\vdots \\ 
\alpha _{k-1}^{\ell }%
\end{array}%
\right) =-\left[ 
\begin{array}{ccc}
\left\vert I_{\limfunc{right}}^{\left( 0\right) }\right\vert _{\mu } & \cdots
& \left\vert I_{\limfunc{right}}^{\left( k-1\right) }\right\vert _{\mu } \\ 
\vdots &  & \vdots \\ 
\left\vert I_{\limfunc{right}}^{\left( k-1\right) }\right\vert _{\mu } & 
\cdots & \left\vert I_{\limfunc{right}}^{\left( 2k-2\right) }\right\vert
_{\mu }%
\end{array}%
\right] \left( 
\begin{array}{c}
\beta _{0}^{\ell } \\ 
\vdots \\ 
\beta _{k-1}^{\ell }%
\end{array}%
\right) ,  \label{Alpert moment k}
\end{equation}%
for $1\leq \ell \leq k$, which we write as%
\begin{equation*}
\boldsymbol{L}_{I}^{k}\mathbf{\alpha }_{k}^{\ell }=-\boldsymbol{R}_{I}^{k}%
\mathbf{\beta }_{k}^{\ell }\ ,
\end{equation*}%
with%
\begin{eqnarray*}
\boldsymbol{L}_{I}^{k} &\equiv &\left[ 
\begin{array}{ccccc}
\left\vert I_{\limfunc{left}}^{\left( 0\right) }\right\vert _{\mu } & 
\left\vert I_{\limfunc{left}}^{\left( 1\right) }\right\vert _{\mu } & \cdots
& \left\vert I_{\limfunc{left}}^{\left( k-2\right) }\right\vert _{\mu } & 
\left\vert I_{\limfunc{left}}^{\left( k-1\right) }\right\vert _{\mu } \\ 
\left\vert I_{\limfunc{left}}^{\left( 1\right) }\right\vert _{\mu } & 
\left\vert I_{\limfunc{left}}^{\left( 2\right) }\right\vert _{\mu } & \cdots
& \left\vert I_{\limfunc{left}}^{\left( k-1\right) }\right\vert _{\mu } & 
\left\vert I_{\limfunc{left}}^{\left( k\right) }\right\vert _{\mu } \\ 
\vdots & \vdots &  &  & \vdots \\ 
\left\vert I_{\limfunc{left}}^{\left( k-2\right) }\right\vert _{\mu } & 
\left\vert I_{\limfunc{left}}^{\left( k-1\right) }\right\vert _{\mu } & 
\cdots & \left\vert I_{\limfunc{left}}^{\left( 2k-4\right) }\right\vert
_{\mu } & \left\vert I_{\limfunc{left}}^{\left( 2k-3\right) }\right\vert
_{\mu } \\ 
\left\vert I_{\limfunc{left}}^{\left( k-1\right) }\right\vert _{\mu } & 
\left\vert I_{\limfunc{left}}^{\left( k\right) }\right\vert _{\mu } & \cdots
& \left\vert I_{\limfunc{left}}^{\left( 2k-3\right) }\right\vert _{\mu } & 
\left\vert I_{\limfunc{left}}^{\left( 2k-2\right) }\right\vert _{\mu }%
\end{array}%
\right] , \\
\boldsymbol{R}_{I}^{k} &\equiv &\left[ 
\begin{array}{ccccc}
\left\vert I_{\limfunc{right}}^{\left( 0\right) }\right\vert _{\mu } & 
\left\vert I_{\limfunc{right}}^{\left( 1\right) }\right\vert _{\mu } & \cdots
& \left\vert I_{\limfunc{right}}^{\left( k-2\right) }\right\vert _{\mu } & 
\left\vert I_{\limfunc{right}}^{\left( k-1\right) }\right\vert _{\mu } \\ 
\left\vert I_{\limfunc{right}}^{\left( 1\right) }\right\vert _{\mu } & 
\left\vert I_{\limfunc{right}}^{\left( 2\right) }\right\vert _{\mu } & \cdots
& \left\vert I_{\limfunc{right}}^{\left( k-1\right) }\right\vert _{\mu } & 
\left\vert I_{\limfunc{right}}^{\left( k\right) }\right\vert _{\mu } \\ 
\vdots & \vdots &  &  & \vdots \\ 
\left\vert I_{\limfunc{right}}^{\left( k-2\right) }\right\vert _{\mu } & 
\left\vert I_{\limfunc{right}}^{\left( k-1\right) }\right\vert _{\mu } & 
\cdots & \left\vert I_{\limfunc{right}}^{\left( 2k-4\right) }\right\vert
_{\mu } & \left\vert I_{\limfunc{right}}^{\left( 2k-3\right) }\right\vert
_{\mu } \\ 
\left\vert I_{\limfunc{right}}^{\left( k-1\right) }\right\vert _{\mu } & 
\left\vert I_{\limfunc{right}}^{\left( k\right) }\right\vert _{\mu } & \cdots
& \left\vert I_{\limfunc{right}}^{\left( 2k-3\right) }\right\vert _{\mu } & 
\left\vert I_{\limfunc{right}}^{\left( 2k-2\right) }\right\vert _{\mu }%
\end{array}%
\right] , \\
\mathbf{\alpha }_{k}^{\ell } &\equiv &\left( 
\begin{array}{c}
\alpha _{0}^{\ell } \\ 
\vdots \\ 
\alpha _{k-1}^{\ell }%
\end{array}%
\right) ,\ \ \ \mathbf{\beta }_{k}^{\ell }=\left( 
\begin{array}{c}
\beta _{0}^{\ell } \\ 
\vdots \\ 
\beta _{k-1}^{\ell }%
\end{array}%
\right)
\end{eqnarray*}%
and by our determinant assumption we can then solve for $\mathbf{\alpha }%
_{k}^{\ell }$ in terms of $\mathbf{\beta }_{k}^{\ell }$; 
\begin{equation}
\mathbf{\alpha }_{k}^{\ell }=-\left[ L_{I}^{k}\right] ^{-1}R_{I}^{k}\ 
\mathbf{\beta }_{k}^{\ell }\ ,\ \ \ \ \ 1\leq \ell \leq k\ .  \label{solve}
\end{equation}%
Thus using the Gram-Schmidt orthogonalization algorithm, there is an
orthonormal basis $\left\{ a_{I}^{\mu ,\ell }\left( x\right) \right\} _{\ell
=1}^{k}$ of $L_{Q;k}^{2}\left( \mu \right) $ consisting of Alpert functions
if and only if at least one of the matrices $\boldsymbol{L}_{I}^{k}$ and $%
\boldsymbol{R}_{I}^{k}$ is nonsingular (equivalently positive definite,
since these symmetric matrices are always nonnegative semidefinite). For
example, if we assume that $L_{I}^{k}$ is invertible, then just as in the
case $k=2$ discussed above, a function $a_{I}^{\mu ,\ell ^{\prime }}\left(
x\right) $ is orthogonal to $a_{I}^{\mu ,\ell }\left( x\right) $ in $%
L^{2}\left( \mu \right) $, i.e.%
\begin{eqnarray*}
0 &=&\int a_{I}^{\mu ,\ell }\left( x\right) a_{I}^{\mu ,\ell ^{\prime
}}\left( x\right) d\mu \left( x\right) \\
&=&\int \left\{ \left( \alpha _{k-1}^{\ell }x^{k-1}+\alpha _{k-2}^{\ell
}x^{k-2}+...+\alpha _{1}^{\ell }x+\alpha _{0}^{\ell }\right) \left( \alpha
_{k-1}^{\ell ^{\prime }}x^{k-1}+\alpha _{k-2}^{\ell ^{\prime
}}x^{k-2}+...+\alpha _{1}^{\ell ^{\prime }}x+\alpha _{0}^{\ell ^{\prime
}}\right) \mathbf{1}_{I_{\limfunc{left}}}\left( x\right) \right. \\
&&+\left. \left( \beta _{k-1}^{\ell }x^{k-1}+\beta _{k-2}^{\ell
}x^{k-2}+...+\beta _{1}^{\ell }x+\beta _{0}^{\ell }\right) \left( \beta
_{k-1}^{\ell ^{\prime }}x^{k-1}+\beta _{k-2}^{\ell ^{\prime
}}x^{k-2}+...+\beta _{1}^{\ell ^{\prime }}x+\beta _{0}^{\ell ^{\prime
}}\right) \mathbf{1}_{I_{\limfunc{right}}}\left( x\right) \right\} d\mu
\left( x\right) \\
&=&\left( \mathbf{\alpha }_{k}^{\ell }\right) ^{\limfunc{tr}}\boldsymbol{L}%
_{I}\mathbf{\alpha }_{k}^{\ell ^{\prime }}+\left( \mathbf{\beta }_{k}^{\ell
}\right) ^{\limfunc{tr}}\boldsymbol{R}_{I}\mathbf{\beta }_{k}^{\ell ^{\prime
}} \\
&=&\left( \mathbf{\beta }_{k}^{\ell }\right) ^{\limfunc{tr}}\left[ 
\boldsymbol{L}_{I}^{-1}\boldsymbol{R}_{I}\right] ^{\limfunc{tr}}\boldsymbol{L%
}_{I}\left[ \boldsymbol{L}_{I}^{-1}\boldsymbol{R}_{I}\right] \mathbf{\beta }%
_{k}^{\ell ^{\prime }}+\left( \mathbf{\beta }_{k}^{\ell }\right) ^{\limfunc{%
tr}}\boldsymbol{R}_{I}\mathbf{\beta }_{k}^{\ell ^{\prime }}
\end{eqnarray*}%
equivalently 
\begin{equation*}
0=\left( \mathbf{\beta }_{k}^{\ell }\right) ^{\limfunc{tr}}\left\{ 
\boldsymbol{R}_{I}\boldsymbol{L}_{I}^{-1}\boldsymbol{L}_{I}\boldsymbol{L}%
_{I}^{-1}\boldsymbol{R}_{I}+\boldsymbol{R}_{I}\right\} \mathbf{\beta }%
_{k}^{\ell ^{\prime }}=\left( \mathbf{\beta }_{k}^{\ell }\right) ^{\limfunc{%
tr}}\boldsymbol{R}_{I}\boldsymbol{L}_{I}^{-1}\left\{ \boldsymbol{R}_{I}+%
\boldsymbol{L}_{I}\right\} \mathbf{\beta }_{k}^{\ell ^{\prime }}=\left( 
\mathbf{\beta }_{k}^{\ell }\right) ^{\limfunc{tr}}\boldsymbol{X}_{I}\mathbf{%
\beta }_{k}^{\ell ^{\prime }},
\end{equation*}%
where $\boldsymbol{X}_{I}\equiv \boldsymbol{L}_{I}\boldsymbol{R}_{I}%
\boldsymbol{L}_{I}^{-1}\left\{ \boldsymbol{R}_{I}+\boldsymbol{L}_{I}\right\} 
$ is invertible since $\boldsymbol{R}_{I}+\boldsymbol{L}_{I}=\boldsymbol{M}%
_{I,k}\succ 0$ by (\ref{Alpert nondeg}). Thus the orthogonality can be
achieved simply by choosing $\mathbf{\beta }_{k}^{\ell ^{\prime }}$
perpendicular to the vector $\left( \mathbf{\beta }_{k}^{\ell }\right) ^{%
\limfunc{tr}}\boldsymbol{X}_{I}$.

Finally we note that the system of equations that we have solved to obtain
an orthonormal basis of Alpert functions is underdetermined since we have
not yet solved for the \emph{additional} moment conditions, 
\begin{equation*}
\int a_{I}^{\mu ,\ell }\left( x\right) x^{i}d\mu \left( x\right) =0,\ \ \ \
\ \text{for all }2\leq \ell \leq k\text{ and }k\leq i\leq k+\ell -2,
\end{equation*}%
and we now explicitly compute the additional equations under which these
additional moment conditions can also be achieved for the orthonormal basis
of Alpert functions. As above have%
\begin{eqnarray*}
0 &=&\alpha _{k-1}^{\ell }\left\vert I_{\limfunc{left}}^{\left( i+k-1\right)
}\right\vert _{\mu }+\alpha _{k-2}^{\ell }\left\vert I_{\limfunc{left}%
}^{\left( i+k-2\right) }\right\vert _{\mu }+...+\alpha _{0}^{\ell
}\left\vert I_{\limfunc{left}}^{\left( i\right) }\right\vert _{\mu } \\
&&+\beta _{k-1}^{\ell }\left\vert I_{\limfunc{right}}^{\left( i+k-1\right)
}\right\vert _{\mu }+\beta _{k-2}^{\ell }\left\vert I_{\limfunc{right}%
}^{\left( i+k-2\right) }\right\vert _{\mu }+...+\beta _{0}^{\ell }\left\vert
I_{\limfunc{right}}^{\left( i\right) }\right\vert _{\mu },
\end{eqnarray*}%
but now for all $2\leq \ell \leq k$ and $k\leq i\leq k+\ell -2$, which leads
to%
\begin{eqnarray*}
&&\left[ 
\begin{array}{ccc}
\left\vert I_{\limfunc{left}}^{\left( i\right) }\right\vert _{\mu } & \cdots
& \left\vert I_{\limfunc{left}}^{\left( i+k-1\right) }\right\vert _{\mu }%
\end{array}%
\right] \left( 
\begin{array}{c}
\alpha _{0}^{\ell } \\ 
\vdots \\ 
\alpha _{k-1}^{\ell }%
\end{array}%
\right) +\left[ 
\begin{array}{ccc}
\left\vert I_{\limfunc{right}}^{\left( i\right) }\right\vert _{\mu } & \cdots
& \left\vert I_{\limfunc{right}}^{\left( i+k-1\right) }\right\vert _{\mu }%
\end{array}%
\right] \left( 
\begin{array}{c}
\beta _{0}^{\ell } \\ 
\vdots \\ 
\beta _{k-1}^{\ell }%
\end{array}%
\right) =0, \\
&&\text{for all }2\leq \ell \leq k\text{ and }k\leq i\leq k+\ell -2,
\end{eqnarray*}%
and can be written as%
\begin{equation}
\mathbf{v}_{\limfunc{left}}^{\left( i+k-1,i\right) }\cdot \mathbf{\alpha }%
_{k}^{\ell }+\mathbf{v}_{\limfunc{right}}^{\left( i+k-1,i\right) }\cdot 
\mathbf{\beta }_{k}^{\ell }=0,\ \ \ 2\leq \ell \leq k\text{ and }k\leq i\leq
k+\ell -2,  \label{written as}
\end{equation}%
where 
\begin{equation*}
\mathbf{v}_{\limfunc{left}}^{\left( i+k-1,i\right) }\equiv \left[ 
\begin{array}{ccc}
& \left\vert I_{\limfunc{left}}^{\left( i\right) }\right\vert _{\mu }\cdots
& \left\vert I_{\limfunc{left}}^{\left( i+k-1\right) }\right\vert _{\mu }%
\end{array}%
\right] ,\ \ \ \mathbf{v}_{\limfunc{right}}^{\left( i+k-1,i\right) }\equiv %
\left[ 
\begin{array}{ccc}
\left\vert I_{\limfunc{right}}^{\left( i\right) }\right\vert _{\mu } & \cdots
& \left\vert I_{\limfunc{right}}^{\left( i+k-1\right) }\right\vert _{\mu }%
\end{array}%
\right] .
\end{equation*}

We now consider the $\left( 
\begin{array}{c}
k \\ 
2%
\end{array}%
\right) =\frac{k\left( k-1\right) }{2}$ orthogonality conditions among the $%
k $ Alpert functions $\left\{ a_{I}^{\mu ,\ell }\right\} _{\ell =1}^{k}$ and
the $\left( 
\begin{array}{c}
k \\ 
2%
\end{array}%
\right) $ additional moment conditions arising from the index choices $%
\left( \ell ,i\right) $ in (\ref{written as}), which for convenience in
visualizing we arrange in triangular form as%
\begin{equation}
\left\{ 
\begin{array}{ccccccc}
\left( k,k\right) & \left( k,k+1\right) & \left( k,k+2\right) & \cdots & 
\left( k,2k-4\right) & \left( k,2k-3\right) & \left( k,2k-2\right) \\ 
\left( k-1,k\right) & \left( k-1,k+1\right) & \left( k-1,k+2\right) & \cdots
& \left( k-1,2k-4\right) & \left( k-1,2k-3\right) &  \\ 
\left( k-2,k\right) & \left( k-2,k+1\right) & \left( k-2,k+2\right) & \cdots
& \left( k-2,2k-4\right) &  &  \\ 
\vdots & \vdots & \vdots & \vdots & \vdots & \vdots & \vdots \\ 
\left( 4,k\right) & \left( 4,k+1\right) & \left( 4,k+2\right) & \cdots &  & 
&  \\ 
\left( 3,k\right) & \left( 3,k+1\right) &  & \cdots &  &  &  \\ 
\left( 2,k\right) &  &  & \cdots &  &  & 
\end{array}%
\right\} .  \label{tri form}
\end{equation}%
Using $\mathbf{\alpha }_{k}^{\ell }=-\left[ \boldsymbol{L}_{I}^{k}\right]
^{-1}\boldsymbol{R}_{I}^{k}\ \mathbf{\beta }_{k}^{\ell }$, the corresponding
additional moment conditions are given by%
\begin{eqnarray*}
&&0=\mathbf{v}_{\limfunc{left}}^{\left( i+k-1,i\right) }\cdot \mathbf{\alpha 
}_{k}^{\ell }+\mathbf{v}_{\limfunc{right}}^{\left( i+k-1,i\right) }\cdot 
\mathbf{\beta }_{k}^{\ell }=\left\{ -\mathbf{v}_{\limfunc{left}}^{\left(
i+k-1,i\right) }\cdot \mathbf{M}_{I}^{k}+\mathbf{v}_{\limfunc{right}%
}^{\left( i+k-1,i\right) }\right\} \cdot \mathbf{\beta }_{k}^{\ell }=\mathbf{%
w}_{k}^{i}\cdot \mathbf{\beta }_{k}^{\ell } \\
&&\ \ \ \ \ \ \ \ \ \ \text{for }2\leq \ell \leq k\text{ and }k\leq i\leq
k+\ell -2,
\end{eqnarray*}%
where $\mathbf{M}_{I}^{k}\equiv -\left[ \boldsymbol{L}_{I}^{k}\right] ^{-1}%
\boldsymbol{R}_{I}^{k}\ \mathbf{\beta }_{k}^{\ell }$, and%
\begin{equation*}
\mathbf{v}_{\limfunc{left}}^{\left( i+k-1,i\right) }\equiv \left[ 
\begin{array}{ccc}
\left\vert I_{\limfunc{left}}^{\left( i\right) }\right\vert _{\mu } & \cdots
& \left\vert I_{\limfunc{left}}^{\left( i+k-1\right) }\right\vert _{\mu }%
\end{array}%
\right] ,\ \ \ \mathbf{v}_{\limfunc{right}}^{\left( i+k-1,i\right) }\equiv %
\left[ 
\begin{array}{ccc}
\left\vert I_{\limfunc{right}}^{\left( i\right) }\right\vert _{\mu } & \cdots
& \left\vert I_{\limfunc{right}}^{\left( i+k-1\right) }\right\vert _{\mu }%
\end{array}%
\right] ,
\end{equation*}%
and%
\begin{equation*}
\mathbf{w}_{k}^{i}\equiv -\mathbf{v}_{\limfunc{left}}^{\left( i+k-1,i\right)
}\cdot \mathbf{M}_{I}^{k}+\mathbf{v}_{\limfunc{right}}^{\left(
i+k-1,i\right) }.
\end{equation*}%
Thus the additional moment conditions are $0=\mathbf{w}_{k}^{i}\cdot \mathbf{%
\beta }_{k}^{\ell }$ for all $2\leq \ell \leq k$ and $k\leq i\leq k+\ell -2$%
, and written in the form of triangular matrix (\ref{tri form}) are%
\begin{eqnarray*}
0 &=&\mathbf{w}_{k}^{k}\cdot \mathbf{\beta }_{k}^{k}=\mathbf{w}%
_{k}^{k+1}\cdot \mathbf{\beta }_{k}^{k}=...=\mathbf{w}_{k}^{2k-2}\cdot 
\mathbf{\beta }_{k}^{k}\ , \\
0 &=&\mathbf{w}_{k}^{k}\cdot \mathbf{\beta }_{k}^{k-1}=\mathbf{w}%
_{k}^{k+1}\cdot \mathbf{\beta }_{k}^{k-1}=...=\mathbf{w}_{k}^{2k-3}\cdot 
\mathbf{\beta }_{k}^{k-1}\ , \\
0 &=&\mathbf{w}_{k}^{k}\cdot \mathbf{\beta }_{k}^{k-2}=\mathbf{w}%
_{k}^{k+1}\cdot \mathbf{\beta }_{k}^{k-2}=...=\mathbf{w}_{k}^{2k-4}\cdot 
\mathbf{\beta }_{k}^{k-2}\ , \\
&&\vdots \\
0 &=&\mathbf{w}_{k}^{k}\cdot \mathbf{\beta }_{k}^{4}=\mathbf{w}%
_{k}^{k+1}\cdot \mathbf{\beta }_{k}^{4}=\mathbf{w}_{k}^{k+2}\cdot \mathbf{%
\beta }_{k}^{4}\ , \\
0 &=&\mathbf{w}_{k}^{k}\cdot \mathbf{\beta }_{k}^{3}=\mathbf{w}%
_{k}^{k+1}\cdot \mathbf{\beta }_{k}^{3}\ , \\
0 &=&\mathbf{w}_{k}^{k}\cdot \mathbf{\beta }_{k}^{2}\ .
\end{eqnarray*}

Now recall that mutual orthogonality of the $k$ Alpert functions $\left\{
a_{I}^{\mu ,\ell }\right\} _{\ell =1}^{k}$ requires choosing $\mathbf{\beta }%
_{k}^{\ell ^{\prime }}$ perpendicular to the vector $\left( \mathbf{\beta }%
_{k}^{\ell }\right) ^{\limfunc{tr}}\boldsymbol{X}_{I}$ for $\ell \neq \ell
^{\prime }$. In addition we must have from (\ref{tri form}) that%
\begin{eqnarray*}
\mathbf{w}_{k}^{k},\mathbf{w}_{k}^{k+1},...,\mathbf{w}_{k}^{2k-2} &\in
&\left( \mathbf{\beta }_{k}^{k}\right) ^{\perp }\text{ }, \\
\mathbf{w}_{k}^{k},\mathbf{w}_{k}^{k+1},...,\mathbf{w}_{k}^{2k-3} &\in
&\left( \mathbf{\beta }_{k}^{k-1}\right) ^{\perp }\text{ }, \\
\mathbf{w}_{k}^{k},\mathbf{w}_{k}^{k+1},...,\mathbf{w}_{k}^{2k-4} &\in
&\left( \mathbf{\beta }_{k}^{k-2}\right) ^{\perp }\text{ }, \\
&&\vdots \\
\mathbf{w}_{k}^{k},\mathbf{w}_{k}^{k+1},\mathbf{w}_{k}^{k+2} &\in &\left( 
\mathbf{\beta }_{k}^{4}\right) ^{\perp }\ , \\
\mathbf{w}_{k}^{k},\mathbf{w}_{k}^{k+1} &\in &\left( \mathbf{\beta }%
_{k}^{3}\right) ^{\perp }\ , \\
\mathbf{w}_{k}^{k} &\in &\left( \mathbf{\beta }_{k}^{2}\right) ^{\perp }\ ,
\end{eqnarray*}%
where $\mathbf{\beta }^{\perp }\equiv \left\{ \mathbf{u}\in \mathbb{R}^{k}:%
\mathbf{u}\cdot \mathbf{\beta }=0\right\} $. Now each of the matrices $%
\boldsymbol{R}_{I}\boldsymbol{L}_{I}^{-1}\mathbf{R}_{I}$ and $\boldsymbol{R}%
_{I}$ is symmetric and positive definite, and so then is their sum $%
\boldsymbol{X}_{I}=\boldsymbol{R}_{I}\boldsymbol{L}_{I}^{-1}\boldsymbol{R}%
_{I}+\boldsymbol{R}_{I}$. Thus we can consider the inner product defined by%
\begin{equation*}
\left\langle \mathbf{v},\mathbf{w}\right\rangle _{I}\equiv \mathbf{v}^{%
\limfunc{tr}}\boldsymbol{X}_{I}\mathbf{w}\ .
\end{equation*}%
This gives us a $k$-dimensional inner product space which we denote by $%
\mathcal{X}_{I}$. Now the condition $\mathbf{w}_{k}^{i}\cdot \mathbf{\beta }%
_{k}^{\ell }=0$ can be written as%
\begin{equation*}
\left\langle \mathbf{u}_{k}^{i},\mathbf{\beta }_{k}^{\ell }\right\rangle
_{I}=\left( \mathbf{u}_{k}^{i}\right) ^{\limfunc{tr}}\boldsymbol{X}_{I}%
\mathbf{\beta }_{k}^{\ell }=\mathbf{w}_{k}^{i}\cdot \mathbf{\beta }%
_{k}^{\ell }=0
\end{equation*}%
if we define $\mathbf{u}_{k}^{i}=\boldsymbol{X}_{I}^{-1}\left( \mathbf{w}%
_{k}^{i}\right) ^{\limfunc{tr}}$. Then our conditions become%
\begin{eqnarray}
\left\langle \mathbf{\beta }_{k}^{\ell },\mathbf{\beta }_{k}^{\ell ^{\prime
}}\right\rangle _{I} &=&0\text{ for }1\leq \ell <\ell ^{\prime }\leq k,
\label{become} \\
\left\langle \mathbf{u}_{k}^{i},\mathbf{\beta }_{k}^{\ell }\right\rangle
_{I} &=&0\ \text{for }2\leq \ell \leq k\text{ and }k\leq i\leq k+\ell -2, 
\notag
\end{eqnarray}%
and the second line above can be written out in triangular form as%
\begin{eqnarray}
\mathbf{u}_{k}^{k},\mathbf{u}_{k}^{k+1},...,\mathbf{u}_{k}^{2k-2} &\in
&\left( \mathbf{\beta }_{k}^{k}\right) ^{\perp _{I}}\text{ },
\label{tri form u} \\
\mathbf{u}_{k}^{k},\mathbf{u}_{k}^{k+1},...,\mathbf{u}_{k}^{2k-3} &\in
&\left( \mathbf{\beta }_{k}^{k-1}\right) ^{\perp _{I}}\text{ },  \notag \\
\mathbf{u}_{k}^{k},\mathbf{u}_{k}^{k+1},...,\mathbf{u}_{k}^{2k-4} &\in
&\left( \mathbf{\beta }_{k}^{k-2}\right) ^{\perp _{I}}\text{ },  \notag \\
&&\vdots  \notag \\
\mathbf{u}_{k}^{k},\mathbf{u}_{k}^{k+1},\mathbf{u}_{k}^{k+2} &\in &\left( 
\mathbf{\beta }_{k}^{4}\right) ^{\perp _{I}}\ ,  \notag \\
\mathbf{u}_{k}^{k},\mathbf{u}_{k}^{k+1} &\in &\left( \mathbf{\beta }%
_{k}^{3}\right) ^{\perp _{I}}\ ,  \notag \\
\mathbf{u}_{k}^{k} &\in &\left( \mathbf{\beta }_{k}^{2}\right) ^{\perp
_{I}}\ ,  \notag
\end{eqnarray}%
where $\mathbf{\beta }^{\perp _{I}}\equiv \left\{ \mathbf{u}\in \mathcal{X}%
_{I}:\left\langle \mathbf{u},\mathbf{\beta }\right\rangle _{I}=0\right\} $.

We can satisfy conditions (\ref{become}) by first choosing a unit vector $%
\mathbf{\beta }_{k}^{k}\in \mathcal{X}_{I}$ so that the first line in (\ref%
{tri form u}) holds. Then we choose a unit vector $\mathbf{\beta }%
_{k}^{k-1}\in \left( \mathbf{\beta }_{k}^{k}\right) ^{\perp _{I}}$ so that
the second line in (\ref{tri form u}) holds. Then we choose a unit vector $%
\mathbf{\beta }_{k}^{k-2}\in \left( \limfunc{Span}\left\{ \mathbf{\beta }%
_{k}^{k-1},\mathbf{\beta }_{k}^{k}\right\} \right) ^{\perp _{I}}$so that the
third line in (\ref{tri form u}) holds. Continuing in this way, we find\
unit vectors $\left\{ \mathbf{\beta }_{k}^{\ell }\right\} _{\ell =1}^{k}$ in
the inner product space $\mathcal{X}_{I}$ so that our conditions (\ref{tri
form u}) and (\ref{become}) hold. For the generic choice of vectors $\mathbf{%
v}_{\limfunc{left}}^{\left( i+k-1,i\right) }$ and $\mathbf{v}_{\limfunc{right%
}}^{\left( i+k-1,i\right) }$, the choice of $\left\{ \mathbf{\beta }%
_{k}^{\ell }\right\} _{\ell =1}^{k}$ will be unique up to sign. This
completes the proof of part (2) of Theorem \ref{dim one}.

\bigskip

We have adopted the additional moment conditions introduced by Alpert in the
setting of Lebesgue measure in \cite{Alp}, but one can in fact replace these
conditions by an essentially arbitrary collection of the correct number of
moment conditions. We leave the nondegenerate case for the reader in the
following exercise, and the degenerate case is not treated here at all,
except for the simple case when $k=2$ and $n=1$ solved below.

\begin{remark}
If $\boldsymbol{M}_{I_{\limfunc{left}},k}\succ 0$, $\boldsymbol{M}_{I_{%
\limfunc{right}},k}\succ 0$ and $\mathcal{M}\subset \left\{ \left( \ell
,i\right) :1\leq \ell \leq k\text{ and }i\geq k\right\} $ has cardinality $%
\left( 
\begin{array}{c}
k \\ 
2%
\end{array}%
\right) $, then we can choose an Alpert basis $\left\{ a_{I}^{\mu ,\ell
}\right\} _{\ell =1}^{k}$ satisfying the additional moment conditions $%
\mathbf{w}_{k}^{i}\cdot \mathbf{\beta }_{k}^{\ell }=0$ for all $\left( \ell
,i\right) \in \mathcal{M}$.
\end{remark}

\subsubsection{The Alpert degenerate case}

In the event that one or more of the matrices $\boldsymbol{L}_{I}^{k}$
and/or $\boldsymbol{R}_{I}^{k}$ is singular, then it is easy to see from (%
\ref{Alpert moment k}) that the maximum number of independent Alpert
functions equals the dimension of the intersection of the ranges of $%
\boldsymbol{L}_{I}^{k}$ and $\boldsymbol{R}_{I}^{k}$ in $\mathbb{R}^{k}$,
i.e.%
\begin{equation*}
\dim L_{Q;k}^{2}\left( \mu \right) =\dim \left( \limfunc{Range}\boldsymbol{L}%
_{I}^{k}\bigcap \limfunc{Range}\boldsymbol{R}_{I}^{k}\right) \ .
\end{equation*}%
In the special case when $k=2$ - when there are only two Alpert functions $%
a_{I}^{\mu ,1}$, $a_{I}^{\mu ,2}$ and just one additional moment condition
for $a_{I}^{\mu ,2}$ - we will show in the next subsection that it is not
always possible to arrange for this additional moment condition to hold. In
fact we will show there that it holds if and only if%
\begin{equation*}
\left( 
\begin{array}{c}
\left\vert I_{\limfunc{right}}^{\left( 2\right) }\right\vert _{\mu } \\ 
\left\vert I_{\limfunc{right}}^{\left( 3\right) }\right\vert _{\mu }%
\end{array}%
\right) \in \limfunc{Range}\boldsymbol{R}_{I}^{2}\ .
\end{equation*}%
Similar results hold for larger $k$, but we will not pursue these here.

\subsection{The special case $n=1$ and $k=2\label{Sub special}$}

First we quickly review and set notation for the nondegenerate case when $%
k=2 $, and later proceed to the degenerate case. Let $\mu $ be a locally
finite positive Borel measure on $\mathbb{R}$ that satisfies the Haar
nondegeneracy condition $\left\vert J\right\vert _{\mu }>0$ for all $J\in 
\mathcal{D}$ and in addition satisfies the Alpert nondegeneracy condition%
\begin{equation}
\det \left[ 
\begin{array}{cc}
\left\vert J^{\left( 0\right) }\right\vert _{\mu } & \left\vert J^{\left(
1\right) }\right\vert _{\mu } \\ 
\left\vert J^{\left( 1\right) }\right\vert _{\mu } & \left\vert J^{\left(
2\right) }\right\vert _{\mu }%
\end{array}%
\right] >0,\ \ \ \ \ \text{for all }J\in \mathcal{D}.  \label{Alpert nondeg}
\end{equation}%
Recall that the determinant in (\ref{Alpert nondeg}) is nonnegative by the
Cauchy-Schwarz inequality, and is positive \emph{if and only if} the
functions $\mathbf{1}_{J}\left( x\right) $ and $x\mathbf{1}_{J}\left(
x\right) $ are linearly independent on $J$:%
\begin{equation*}
\left\vert J^{\left( 1\right) }\right\vert _{\mu }^{2}=\left( \int_{J}xd\mu
\left( x\right) \right) ^{2}\leq \left( \int_{J}1^{2}d\mu \left( x\right)
\right) \left( \int_{J}x^{2}d\mu \left( x\right) \right) =\left\vert
J^{\left( 0\right) }\right\vert _{\mu }\left\vert J^{\left( 2\right)
}\right\vert _{\mu }
\end{equation*}%
with equality if and only if the functions $\mathbf{1}_{J}\left( x\right) $
and $x\mathbf{1}_{J}\left( x\right) $ are linearly dependent on $J$.

We require the functions%
\begin{eqnarray*}
a_{I}^{\mu ,1}\left( x\right) &\equiv &\left( \alpha _{1}^{1}x+\alpha
_{0}^{1}\right) \mathbf{1}_{I_{\limfunc{left}}}\left( x\right) +\left( \beta
_{1}^{1}x+\beta _{0}^{1}\right) \mathbf{1}_{I_{\limfunc{right}}}\left(
x\right) , \\
a_{I}^{\mu ,2}\left( x\right) &\equiv &\left( \alpha _{1}^{2}x+\alpha
_{0}^{2}\right) \mathbf{1}_{I_{\limfunc{left}}}\left( x\right) +\left( \beta
_{1}^{2}x+\beta _{0}^{2}\right) \mathbf{1}_{I_{\limfunc{right}}}\left(
x\right) ,
\end{eqnarray*}%
to satisfy the moment properties,%
\begin{eqnarray*}
&&\left( A\right) \ \int a_{I}^{\mu ,1}\left( x\right) d\mu \left( x\right)
=0\text{ and }\int a_{I}^{\mu ,1}\left( x\right) xd\mu \left( x\right) =0, \\
&&\left( B\right) \ \int a_{I}^{\mu ,2}\left( x\right) d\mu \left( x\right)
=0\text{ and }\int a_{I}^{\mu ,2}\left( x\right) xd\mu \left( x\right) =0, \\
&&\left( C\right) \ \int a_{I}^{\mu ,2}\left( x\right) x^{2}d\mu \left(
x\right) =0,
\end{eqnarray*}%
the orthogonality property,%
\begin{equation*}
(D)\ \int a_{I}^{\mu ,1}\left( x\right) a_{I}^{\mu ,2}\left( x\right) d\mu
\left( x\right) =0,
\end{equation*}%
and the normalization properties, 
\begin{eqnarray*}
\left( \int \left\vert a_{I}^{\mu ,1}\left( x\right) \right\vert ^{2}d\mu
\left( x\right) \right) ^{\frac{1}{2}} &=&1, \\
\left( \int \left\vert a_{I}^{\mu ,2}\left( x\right) \right\vert ^{2}d\mu
\left( x\right) \right) ^{\frac{1}{2}} &=&1.
\end{eqnarray*}%
Theorem \ref{main1} above gives the following conclusion.

\begin{theorem}
The collection 
\begin{equation*}
\mathcal{U}^{\limfunc{Alpert},\mu }\equiv \left\{ a_{I}^{\mu ,1},a_{I}^{\mu
,2}\right\} _{I\in \mathcal{D}}
\end{equation*}%
is an orthonormal basis for $L^{2}\left( \mu \right) $.
\end{theorem}

If we wish to include the additional moment condition (C), i.e.$\ \int
a_{I}^{\mu ,2}\left( x\right) x^{2}d\mu \left( x\right) =0$, then we must in
addition solve%
\begin{eqnarray*}
0 &=&\int \left\{ \left( \alpha _{1}^{2}x+\alpha _{0}^{2}\right) \mathbf{1}%
_{I_{\limfunc{left}}}\left( x\right) +\left( \beta _{1}^{2}x+\beta
_{0}^{2}\right) \mathbf{1}_{I_{\limfunc{right}}}\left( x\right) \right\}
x^{2}d\mu \left( x\right) \\
&=&\alpha _{0}^{2}\left\vert I_{\limfunc{left}}^{\left( 2\right)
}\right\vert _{\mu }+\alpha _{1}^{2}\left\vert I_{\limfunc{left}}^{\left(
3\right) }\right\vert _{\mu }+\beta _{0}^{2}\left\vert I_{\limfunc{right}%
}^{\left( 2\right) }\right\vert _{\mu }+\beta _{1}^{2}\left\vert I_{\limfunc{%
right}}^{\left( 3\right) }\right\vert _{\mu } \\
&=&\left( 
\begin{array}{cc}
\left\vert I_{\limfunc{left}}^{\left( 2\right) }\right\vert _{\mu } & 
\left\vert I_{\limfunc{left}}^{\left( 3\right) }\right\vert _{\mu }%
\end{array}%
\right) \left( 
\begin{array}{c}
\alpha _{0}^{2} \\ 
\alpha _{1}^{2}%
\end{array}%
\right) +\left( 
\begin{array}{cc}
\left\vert I_{\limfunc{right}}^{\left( 2\right) }\right\vert _{\mu } & 
\left\vert I_{\limfunc{right}}^{\left( 3\right) }\right\vert _{\mu }%
\end{array}%
\right) \left( 
\begin{array}{c}
\beta _{0}^{2} \\ 
\beta _{1}^{2}%
\end{array}%
\right) \\
&=&\left\{ -\left( 
\begin{array}{cc}
\left\vert I_{\limfunc{left}}^{\left( 2\right) }\right\vert _{\mu } & 
\left\vert I_{\limfunc{left}}^{\left( 3\right) }\right\vert _{\mu }%
\end{array}%
\right) L_{I}^{-1}R_{I}+\left( 
\begin{array}{cc}
\left\vert I_{\limfunc{right}}^{\left( 2\right) }\right\vert _{\mu } & 
\left\vert I_{\limfunc{right}}^{\left( 3\right) }\right\vert _{\mu }%
\end{array}%
\right) \right\} \left( 
\begin{array}{c}
\beta _{0}^{2} \\ 
\beta _{1}^{2}%
\end{array}%
\right)
\end{eqnarray*}%
which means we must choose $\left( 
\begin{array}{c}
\beta _{0}^{2} \\ 
\beta _{1}^{2}%
\end{array}%
\right) $ perpendicular to the vector 
\begin{equation*}
-\left( 
\begin{array}{cc}
\left\vert I_{\limfunc{left}}^{\left( 2\right) }\right\vert _{\mu } & 
\left\vert I_{\limfunc{left}}^{\left( 3\right) }\right\vert _{\mu }%
\end{array}%
\right) \boldsymbol{L}_{I}^{-1}\boldsymbol{R}_{I}+\left( 
\begin{array}{cc}
\left\vert I_{\limfunc{right}}^{\left( 2\right) }\right\vert _{\mu } & 
\left\vert I_{\limfunc{right}}^{\left( 3\right) }\right\vert _{\mu }%
\end{array}%
\right) .
\end{equation*}%
Thus it must be the case that the two vectors%
\begin{equation*}
\left( 
\begin{array}{cc}
\beta _{0}^{1} & \beta _{1}^{1}%
\end{array}%
\right) \boldsymbol{R}_{I}\boldsymbol{L}_{I}^{-1}\left\{ \boldsymbol{R}_{I}+%
\boldsymbol{L}_{I}\right\}
\end{equation*}%
and%
\begin{equation*}
-\left( 
\begin{array}{cc}
\left\vert I_{\limfunc{left}}^{\left( 2\right) }\right\vert _{\mu } & 
\left\vert I_{\limfunc{left}}^{\left( 3\right) }\right\vert _{\mu }%
\end{array}%
\right) L_{I}^{-1}R_{I}+\left( 
\begin{array}{cc}
\left\vert I_{\limfunc{right}}^{\left( 2\right) }\right\vert _{\mu } & 
\left\vert I_{\limfunc{right}}^{\left( 3\right) }\right\vert _{\mu }%
\end{array}%
\right)
\end{equation*}%
are parallel. But since $\boldsymbol{R}_{I}\boldsymbol{L}_{I}^{-1}\left\{ 
\boldsymbol{R}_{I}+\boldsymbol{L}_{I}\right\} $ is invertible, this can
clearly be achieved by choosing $\left( 
\begin{array}{cc}
\beta _{0}^{1} & \beta _{1}^{1}%
\end{array}%
\right) $ appropriately, thereby using up our last degree of freedom in the
case $k=2$.

Finally, we examine what happens when one or more of the nondegeneracy
conditions (\ref{Haar nondeg}) and (\ref{Alpert nondeg}) fails. Note that
for a given interval $J$, we have that (\ref{Haar nondeg}) holds and (\ref%
{Alpert nondeg}) fails \emph{if and only if} $\mu 1_{J}$ is a point mass.
Indeed, $\mathbf{1}_{J}\left( x\right) $ and $x\mathbf{1}_{J}\left( x\right) 
$ are dependent if and only if $\mu 1_{J}$ is a point mass, which we locate
at $x_{J}\in J$. In this case we set%
\begin{equation*}
\mu 1_{J}=\left\vert J\right\vert _{\mu }\delta _{x_{J}}\ .
\end{equation*}%
We also recall%
\begin{equation*}
\boldsymbol{M}_{J}=\boldsymbol{M}_{J,2}\equiv \left[ 
\begin{array}{cc}
\left\vert J_{\limfunc{left}}^{\left( 0\right) }\right\vert _{\mu } & 
\left\vert J_{\limfunc{left}}^{\left( 1\right) }\right\vert _{\mu } \\ 
\left\vert J_{\limfunc{left}}^{\left( 1\right) }\right\vert _{\mu } & 
\left\vert J_{\limfunc{left}}^{\left( 2\right) }\right\vert _{\mu }%
\end{array}%
\right] \text{ for }J\in \mathcal{D}.
\end{equation*}

\begin{lemma}
Let $I=I_{\limfunc{left}}\dot{\cup}I_{\limfunc{right}}$ be the decomposition
of $I$ into its two children.

\begin{enumerate}
\item If (\ref{Haar nondeg}) holds for both $I_{\limfunc{left}}$ and $I_{%
\limfunc{right}}$, and (\ref{Alpert nondeg}) holds for $I_{\limfunc{left}}$
but fails for $I_{\limfunc{right}}$, then $a_{I}^{\mu ,1}=a_{I}^{\mu ,2}\neq
0$ in $L^{2}\left( \mu \right) $. This conclusion persists in the opposite
situation where (\ref{Haar nondeg}) holds for both $I_{\limfunc{left}}$ and $%
I_{\limfunc{right}}$, but (\ref{Alpert nondeg}) fails for $I_{\limfunc{left}%
} $ and holds for $I_{\limfunc{right}}$.

\item In all other degenerate cases, where at least one of the matrices $%
\left[ 
\begin{array}{cc}
\left\vert \left( I_{\limfunc{left}}\right) ^{\left( 0\right) }\right\vert
_{\mu } & \left\vert \left( I_{\limfunc{left}}\right) ^{\left( 1\right)
}\right\vert _{\mu } \\ 
\left\vert \left( I_{\limfunc{left}}\right) ^{\left( 1\right) }\right\vert
_{\mu } & \left\vert \left( I_{\limfunc{left}}\right) ^{\left( 2\right)
}\right\vert _{\mu }%
\end{array}%
\right] $, $\left[ 
\begin{array}{cc}
\left\vert \left( I_{\limfunc{right}}\right) ^{\left( 0\right) }\right\vert
_{\mu } & \left\vert \left( I_{\limfunc{right}}\right) ^{\left( 1\right)
}\right\vert _{\mu } \\ 
\left\vert \left( I_{\limfunc{right}}\right) ^{\left( 1\right) }\right\vert
_{\mu } & \left\vert \left( I_{\limfunc{right}}\right) ^{\left( 2\right)
}\right\vert _{\mu }%
\end{array}%
\right] $ fails to be positive definite, we have $a_{I}^{\mu ,1}=a_{I}^{\mu
,2}=0$ in $L^{2}\left( \mu \right) $.
\end{enumerate}
\end{lemma}

\begin{proof}
Fix $I\in \mathcal{D}$.

Assertion (\textbf{1}): We have%
\begin{equation*}
\det \left[ 
\begin{array}{cc}
\left\vert \left( I_{\limfunc{left}}\right) ^{\left( 0\right) }\right\vert
_{\mu } & \left\vert \left( I_{\limfunc{left}}\right) ^{\left( 1\right)
}\right\vert _{\mu } \\ 
\left\vert \left( I_{\limfunc{left}}\right) ^{\left( 1\right) }\right\vert
_{\mu } & \left\vert \left( I_{\limfunc{left}}\right) ^{\left( 2\right)
}\right\vert _{\mu }%
\end{array}%
\right] >0\text{ and }\det \left[ 
\begin{array}{cc}
\left\vert \left( I_{\limfunc{right}}\right) ^{\left( 0\right) }\right\vert
_{\mu } & \left\vert \left( I_{\limfunc{right}}\right) ^{\left( 1\right)
}\right\vert _{\mu } \\ 
\left\vert \left( I_{\limfunc{right}}\right) ^{\left( 1\right) }\right\vert
_{\mu } & \left\vert \left( I_{\limfunc{right}}\right) ^{\left( 2\right)
}\right\vert _{\mu }%
\end{array}%
\right] =0.
\end{equation*}%
Thus $\mathbf{1}_{I_{j,\limfunc{right}}^{m}}\mu $ is a point mass located at 
$x_{I_{j,\limfunc{right}}^{m}}$. Now (\ref{solve}) shows that $\left( 
\begin{array}{c}
\alpha _{0}^{\ell } \\ 
\alpha _{1}^{\ell }%
\end{array}%
\right) $ is uniquely determined by $\left( 
\begin{array}{c}
\beta _{0}^{\ell } \\ 
\beta _{1}^{\ell }%
\end{array}%
\right) $ for $1\leq \ell \leq 2$, and we know that the range of $\left[ 
\begin{array}{cc}
\left\vert \left( I_{j,\limfunc{right}}^{m}\right) ^{\left( 2\right)
}\right\vert _{\mu } & \left\vert \left( I_{j,\limfunc{right}}^{m}\right)
^{\left( 1\right) }\right\vert _{\mu } \\ 
\left\vert \left( I_{j,\limfunc{right}}^{m}\right) ^{\left( 1\right)
}\right\vert _{\mu } & \left\vert \left( I_{j,\limfunc{right}}^{m}\right)
^{\left( 0\right) }\right\vert _{\mu }%
\end{array}%
\right] $ is just one-dimensional, since $\mathbf{1}_{I_{j,\limfunc{right}%
}^{m}}\mu $ is a point mass located at $x_{I_{j,\limfunc{right}}^{m}}$. It
follows that both $a_{I}^{\mu ,1}$ and $a_{I}^{\mu ,2}$ are constant on the
interval $I_{j,\limfunc{right}}^{m}$ in the space $L^{2}\left( \mu \right) $%
, and since the null space of $\left[ 
\begin{array}{cc}
\left\vert \left( I_{j,\limfunc{right}}^{m}\right) ^{\left( 0\right)
}\right\vert _{\mu } & \left\vert \left( I_{j,\limfunc{right}}^{m}\right)
^{\left( 1\right) }\right\vert _{\mu } \\ 
\left\vert \left( I_{j,\limfunc{right}}^{m}\right) ^{\left( 1\right)
}\right\vert _{\mu } & \left\vert \left( I_{j,\limfunc{right}}^{m}\right)
^{\left( 2\right) }\right\vert _{\mu }%
\end{array}%
\right] $ is one-dimensional, namely $\limfunc{Span}\left( 
\begin{array}{c}
-x_{I_{j,\limfunc{right}}^{m}} \\ 
1%
\end{array}%
\right) $, we see that the functions $a_{I}^{\mu ,1}$ and $a_{I}^{\mu ,2}$
are linearly dependent, i.e. there is only one Alpert function in this case,
i.e. $a_{I_{j}^{m}}^{\mu ,1}=a_{I_{j}^{m}}^{\mu ,2}$.

Assertion (\textbf{2}): From above we have $\dim L_{Q;2}^{2}\left( \mu
\right) =\dim \left( \limfunc{Range}\boldsymbol{L}_{I}^{2}\bigcap \limfunc{%
Range}\boldsymbol{R}_{I}^{2}\right) $. Thus if one of the ranges is $\left\{
0\right\} $ we are done. So we are left with the case where (\ref{Haar
nondeg}) holds for both $I_{j,\limfunc{left}}^{m}$ and $I_{j,\limfunc{right}%
}^{m}$, and (\ref{Alpert nondeg}) fails for both $I_{j,\limfunc{left}}^{m}$
and $I_{j,\limfunc{right}}^{m}$. We then have%
\begin{equation*}
\left[ 
\begin{array}{cc}
\left\vert \left( I_{j,\limfunc{left}}^{m}\right) ^{\left( 0\right)
}\right\vert _{\mu } & \left\vert \left( I_{j,\limfunc{left}}^{m}\right)
^{\left( 1\right) }\right\vert _{\mu } \\ 
\left\vert \left( I_{j,\limfunc{left}}^{m}\right) ^{\left( 1\right)
}\right\vert _{\mu } & \left\vert \left( I_{j,\limfunc{left}}^{m}\right)
^{\left( 2\right) }\right\vert _{\mu }%
\end{array}%
\right] \left( 
\begin{array}{c}
\alpha _{0}^{1} \\ 
\alpha _{1}^{1}%
\end{array}%
\right) =-\left[ 
\begin{array}{cc}
\left\vert \left( I_{j,\limfunc{right}}^{m}\right) ^{\left( 0\right)
}\right\vert _{\mu } & \left\vert \left( I_{j,\limfunc{right}}^{m}\right)
^{\left( 1\right) }\right\vert _{\mu } \\ 
\left\vert \left( I_{j,\limfunc{right}}^{m}\right) ^{\left( 1\right)
}\right\vert _{\mu } & \left\vert \left( I_{j,\limfunc{right}}^{m}\right)
^{\left( 2\right) }\right\vert _{\mu }%
\end{array}%
\right] \left( 
\begin{array}{c}
\beta _{0}^{1} \\ 
\beta _{1}^{1}%
\end{array}%
\right) .
\end{equation*}%
Now the point mass $\mathbf{1}_{I_{\limfunc{right}}}\mu $ is located at $%
x_{I_{j,\limfunc{right}}^{m}}$, and the point mass $\mathbf{1}_{I_{\limfunc{%
left}}}\mu $ is located at $x_{I_{j,\limfunc{left}}^{m}}$, and thus the
one-dimensional ranges of $\left[ 
\begin{array}{cc}
\left\vert \left( I_{j,\limfunc{left}}^{m}\right) ^{\left( 0\right)
}\right\vert _{\mu } & \left\vert \left( I_{j,\limfunc{left}}^{m}\right)
^{\left( 1\right) }\right\vert _{\mu } \\ 
\left\vert \left( I_{j,\limfunc{left}}^{m}\right) ^{\left( 1\right)
}\right\vert _{\mu } & \left\vert \left( I_{j,\limfunc{left}}^{m}\right)
^{\left( 2\right) }\right\vert _{\mu }%
\end{array}%
\right] $ and $\left[ 
\begin{array}{cc}
\left\vert \left( I_{j,\limfunc{right}}^{m}\right) ^{\left( 0\right)
}\right\vert _{\mu } & \left\vert \left( I_{j,\limfunc{right}}^{m}\right)
^{\left( 1\right) }\right\vert _{\mu } \\ 
\left\vert \left( I_{j,\limfunc{right}}^{m}\right) ^{\left( 1\right)
}\right\vert _{\mu } & \left\vert \left( I_{j,\limfunc{right}}^{m}\right)
^{\left( 2\right) }\right\vert _{\mu }%
\end{array}%
\right] $ are respectively spanned by the vectors $\left( 
\begin{array}{c}
x_{I_{j,\limfunc{left}}^{m}} \\ 
1%
\end{array}%
\right) $ and $\left( 
\begin{array}{c}
x_{I_{j,\limfunc{right}}^{m}} \\ 
1%
\end{array}%
\right) $, which are independent since $I_{j,\limfunc{left}}^{m}\cap I_{j,%
\limfunc{right}}^{m}=\emptyset $. Thus only the trivial solution$\ \left( 
\begin{array}{c}
\alpha _{0}^{1} \\ 
\alpha _{1}^{1}%
\end{array}%
\right) =\left( 
\begin{array}{c}
\beta _{0}^{1} \\ 
\beta _{1}^{1}%
\end{array}%
\right) =\left( 
\begin{array}{c}
0 \\ 
0%
\end{array}%
\right) $ exists, and it follows that both $a_{I_{j}^{m}}^{\mu ,1}$ and $%
a_{I_{j}^{m}}^{\mu ,2}$ vanish in $L^{2}\left( \mu \right) $.
\end{proof}

The resulting pared collection 
\begin{equation*}
\mathcal{U}^{\limfunc{Alpert},\mu }\equiv \left\{ a_{I}^{\mu ,1},a_{I}^{\mu
,2}\right\} _{I\in \mathcal{D}}\ ,
\end{equation*}%
where $a_{I}^{\mu ,1}$ or $a_{I}^{\mu ,2}$ or both are removed according to
the lemma above, is an orthonormal basis for $L^{2}\left( \mu \right) $.

Finally, we consider the additional moment condition (C) in the case where
assertion (1) of the above lemma holds, namely when (\ref{Haar nondeg})
holds for both $I_{j,\limfunc{left}}^{m}$ and $I_{j,\limfunc{right}}^{m}$,
and (\ref{Alpert nondeg}) holds for $I_{j,\limfunc{left}}^{m}$ but fails for 
$I_{j,\limfunc{right}}^{m}$. In this case $a_{I_{j}^{m}}^{\mu
}=b_{I_{j}^{m}}^{\mu }\neq 0$ in $L^{2}\left( \mu \right) $, and this is the
only case in which there is just one Alpert function (apart from the
symmetric case when (\ref{Alpert nondeg}) holds for $I_{j,\limfunc{right}%
}^{m}$ but fails for $I_{j,\limfunc{left}}^{m}$).

From the calculations above we must have that the two vectors%
\begin{equation*}
\left( 
\begin{array}{cc}
\beta _{0}^{1} & \beta _{1}^{1}%
\end{array}%
\right) \boldsymbol{R}_{I}\boldsymbol{L}_{I}^{-1}\left\{ \boldsymbol{R}_{I}+%
\boldsymbol{L}_{I}\right\}
\end{equation*}%
and%
\begin{equation*}
-\left( 
\begin{array}{cc}
\left\vert I_{\limfunc{left}}^{\left( 2\right) }\right\vert _{\mu } & 
\left\vert I_{\limfunc{left}}^{\left( 3\right) }\right\vert _{\mu }%
\end{array}%
\right) \boldsymbol{L}_{I}^{-1}\boldsymbol{R}_{I}+\left( 
\begin{array}{cc}
\left\vert I_{\limfunc{right}}^{\left( 2\right) }\right\vert _{\mu } & 
\left\vert I_{\limfunc{right}}^{\left( 3\right) }\right\vert _{\mu }%
\end{array}%
\right)
\end{equation*}%
are parallel. But since $\boldsymbol{R}_{I}\boldsymbol{L}_{I}^{-1}\left\{ 
\boldsymbol{R}_{I}+\boldsymbol{L}_{I}\right\} $ is no longer invertible -
indeed, $\boldsymbol{R}_{I}\boldsymbol{L}_{I}^{-1}\left\{ \boldsymbol{R}_{I}+%
\boldsymbol{L}_{I}\right\} $ has rank $1$ since $\boldsymbol{R}_{I}$ does -
we cannot necessarily solve 
\begin{equation*}
\left( 
\begin{array}{cc}
\beta _{0}^{1} & \beta _{1}^{1}%
\end{array}%
\right) \left\{ \boldsymbol{R}_{I}\boldsymbol{L}_{I}^{-1}\boldsymbol{R}_{I}+%
\boldsymbol{R}_{I}\right\} =\lambda \left\{ \left( 
\begin{array}{cc}
\left\vert I_{\limfunc{right}}^{\left( 2\right) }\right\vert _{\mu } & 
\left\vert I_{\limfunc{right}}^{\left( 3\right) }\right\vert _{\mu }%
\end{array}%
\right) -\left( 
\begin{array}{cc}
\left\vert I_{\limfunc{left}}^{\left( 2\right) }\right\vert _{\mu } & 
\left\vert I_{\limfunc{left}}^{\left( 3\right) }\right\vert _{\mu }%
\end{array}%
\right) \boldsymbol{L}_{I}^{-1}\boldsymbol{R}_{I}\right\} .
\end{equation*}%
In fact, there is a solution if and only if the transpose of the vector 
\begin{equation*}
\left( 
\begin{array}{cc}
\left\vert I_{\limfunc{right}}^{\left( 2\right) }\right\vert _{\mu } & 
\left\vert I_{\limfunc{right}}^{\left( 3\right) }\right\vert _{\mu }%
\end{array}%
\right) -\left( 
\begin{array}{cc}
\left\vert I_{\limfunc{left}}^{\left( 2\right) }\right\vert _{\mu } & 
\left\vert I_{\limfunc{left}}^{\left( 3\right) }\right\vert _{\mu }%
\end{array}%
\right) \boldsymbol{L}_{I}^{-1}\boldsymbol{R}_{I}
\end{equation*}%
lies in the one-dimensional range of the matrix $R_{I}$, which in turn holds
if and only if%
\begin{equation*}
\left( 
\begin{array}{c}
\left\vert I_{\limfunc{right}}^{\left( 2\right) }\right\vert _{\mu } \\ 
\left\vert I_{\limfunc{right}}^{\left( 3\right) }\right\vert _{\mu }%
\end{array}%
\right) \in \limfunc{Range}\boldsymbol{R}_{I}\ .
\end{equation*}

\begin{remark}
The authors thank Fletcher Gates for showing them that, in the case under
consideration, this latter condition is always satisfied.
\end{remark}

\section{Application: a two weight $Tp$ Conjecture}

Using the weighted Alpert wavelet bases constructed in Theorem \ref{main1},
we can formulate an associated $Tp$-type theorem in dimension $n=1$ where
testing over indicators is replaced by testing over indicators times
polynomials of degree at most $k-1$, and the energy condition is replaced by
an associated $k$-energy condition. Unfortunately, at this point in time, we
cannot demonstrate that Conjecture \ref{TP} below produces new interesting
weighted inequalities, despite that fact that we provide an example to show
that the $k$-energy condition is strictly weaker than the usual energy
condition, even in the presence of the Muckenhoupt condition. But see
Subsection \ref{CZ test} for a demonstration that Conjecture \ref{TP}
differs `logically' from existing $T1$-type theorems in the literature.

Let $0\leq \alpha <1$, $k\in \mathbb{N}$ and $0<\delta <1$. We define a $%
\left( k+\delta \right) $-smooth $\alpha $-fractional CZ kernel $K^{\alpha
}(x,y)$ to be a real-valued function defined on $\mathbb{R}\times \mathbb{R}$
satisfying the following fractional size and smoothness conditions of order $%
1+\delta $: For $x\neq y$,%
\begin{eqnarray}
&&\left\vert K^{\alpha }\left( x,y\right) \right\vert \leq C_{CZ}\left\vert
x-y\right\vert ^{\alpha -1}\text{ and }\left\vert \nabla ^{j}K^{\alpha
}\left( x,y\right) \right\vert \leq C_{CZ}\left\vert x-y\right\vert ^{\alpha
-1-j},\ \ \ 1\leq j\leq k,  \label{sizeandsmoothness'} \\
&&\left\vert \nabla ^{k}K^{\alpha }\left( x,y\right) -\nabla ^{k}K^{\alpha
}\left( x^{\prime },y\right) \right\vert \leq C_{CZ}\left( \frac{\left\vert
x-x^{\prime }\right\vert }{\left\vert x-y\right\vert }\right) ^{\delta
}\left\vert x-y\right\vert ^{\alpha -1-k},\ \ \ \ \ \frac{\left\vert
x-x^{\prime }\right\vert }{\left\vert x-y\right\vert }\leq \frac{1}{2}, 
\notag
\end{eqnarray}%
and the last inequality also holds for the adjoint kernel in which $x$ and $%
y $ are interchanged. We associate a corresponding Calder\'{o}n-Zygmund
operator $T^{\alpha }$ is the usual way (see e.g. \cite{LaWi} or \cite%
{SaShUr7}).

The following conjectured $Tp$ theorem with an energy side condition differs
from the corresponding $T1$ theorem with an energy side condition in two
ways:

\begin{enumerate}
\item Because of the stronger moment conditions satisfied Alpert wavelets,
the usual energy condition assumption from the $T1$ theorem can be weakened
in the $Tp$ theorem.

\item Due to the weaker telescoping identities satisfied by Alpert wavelets,
the $T1$ testing conditions must be strengthened to testing polynomials
times indicators.
\end{enumerate}

\begin{conjecture}
\label{TP}Let $0\leq \alpha <1$ and $0<\delta <1$. Let $T^{\alpha }$ be a $%
\left( k+\delta \right) $-smooth $\alpha $-fractional Calder\'{o}n-Zygmund
operator on the real line. Suppose $\sigma $ and $\omega $ are locally
finite positive Borel measures on $\mathbb{R}$ that satisfy the $k$-energy
condition%
\begin{equation}
\left( \mathcal{E}_{2,k}^{\alpha }\right) ^{2}\equiv \sup_{I=\dot{\cup}%
_{r=1}^{\infty }I_{r}}\frac{1}{\left\vert I\right\vert _{\sigma }}%
\sum_{r=1}^{\infty }\left( \frac{\mathrm{P}_{k}^{\alpha }\left( I_{r},%
\mathbf{1}_{I}\sigma \right) }{\left\vert I_{r}\right\vert ^{k}}\right)
^{2}\left\Vert \left( x-m_{I_{r}}^{k}\right) ^{k}\right\Vert _{L^{2}\left( 
\mathbf{1}_{I_{r}}\omega \right) }^{2},  \label{k energy}
\end{equation}%
as well as the dual $k$-energy condition obtained by interchanging the
measures $\sigma $ and $\omega $. Then the operator $T_{\sigma }^{\alpha
}f\equiv T^{\alpha }\left( f\sigma \right) $ is bounded from $L^{2}\left(
\sigma \right) $ to $L^{2}\left( \omega \right) $ (in the sense that tangent
line truncations are uniformly bounded by a constant $\mathfrak{N}%
_{T^{\alpha }}\left( \sigma ,\omega \right) $) if

\begin{enumerate}
\item there is a positive constant $\mathfrak{T}_{T^{\alpha }}^{k}\left(
\sigma ,\omega \right) $ such that 
\begin{equation*}
\int_{Q}\left\vert T_{\sigma }^{\alpha }\mathbf{1}_{Q}p\right\vert
^{2}d\omega \leq \left( \mathfrak{T}_{T^{\alpha }}^{k}\left( \sigma ,\omega
\right) \right) ^{2}\int_{Q}\left\vert p\right\vert ^{2}d\sigma ,
\end{equation*}%
for all intervals $Q$ and polynomials $p\left( x\right)
=c_{0}+c_{1}x+...+c_{k-1}x^{k-1}$ of degree at most $k-1$, as well as the
dual testing conditions obtained by interchanging the measures $\sigma $ and 
$\omega $ and replacing $T_{\sigma }^{\alpha }$ with its dual $T_{\omega
}^{\alpha }$; and

\item the fractional Muckenhoupt condition is finite: $\mathfrak{A}%
_{2}^{\alpha }\left( \sigma ,\omega \right) <\infty $.
\end{enumerate}
\end{conjecture}

A proof of this conjecture would follow the somewhat standard lines of proof
for $T1$-type theorems already in the literature (\cite{NTV4}, \cite%
{LaSaShUr3}, \cite{Lac}, \cite{Hyt2}), namely an inner product $\left\langle
T_{\sigma }^{\alpha }f,g\right\rangle _{L^{2}\left( \omega \right) }$ is
expanded in Alpert projections as%
\begin{equation}
\left\langle T_{\sigma }^{\alpha }f,g\right\rangle _{L^{2}\left( \omega
\right) }=\sum_{\substack{ I\in \mathcal{D} \\ 1\leq \ell \leq k}}\sum
_{\substack{ J\in \mathcal{G} \\ 1\leq \ell ^{\prime }\leq k^{\prime }}}%
\left\langle T_{\sigma }^{\alpha }\bigtriangleup _{I}^{\sigma ,\ell
}f,\bigtriangleup _{J}^{\omega ,\ell ^{\prime }}g\right\rangle _{L^{2}\left(
\omega \right) }\ ,  \label{inner pdt sum}
\end{equation}%
and then decomposed into many separate infinite sums according to the
relative sizes, locations and goodness of the intervals $I$ and $J$, which
are then all controlled differently. We first highlight the two main points
of departure in controlling these different sums, followed by a brief
description of the sums themselves and how they are handled in the $Tp\,$%
situation, as well as pointing to an obstruction to the proof.

\begin{itemize}
\item The estimate for norms of Alpert projections$\left\Vert \bigtriangleup
_{J}^{\omega }T^{\alpha }\mu \right\Vert _{L^{2}\left( \omega \right) }^{2}$%
, called the Monotonicity Lemma below and in \cite{LaWi}, \cite{SaShUr7}, is
improved by the extra vanishing moments of Alpert wavelets to 
\begin{equation*}
\left\Vert \bigtriangleup _{J}^{\omega }T^{\alpha }\mu \right\Vert
_{L^{2}\left( \omega \right) }^{2}\lesssim \left( \frac{\mathrm{P}%
_{k}^{\alpha }\left( J,\mu \right) }{\left\vert J\right\vert ^{k}}\right)
^{2}\left\Vert \left( x-m_{J}^{k}\right) ^{k}\right\Vert _{L^{2}\left( 
\mathbf{1}_{J}\omega \right) }^{2}\ ,
\end{equation*}%
which in turn can then be controlled by the $k$-energy condition (\ref{k
energy}) above, \textbf{weaker} than the usual energy condition with $k=1$;

\item The telescoping identities (\ref{telescoping}) reduce sums of
consecutive Alpert projections $\bigtriangleup _{I;k}^{\mu }$ to projections 
$\mathbb{E}_{Q;k}^{\mu }$ onto spaces of polynomials of degree at most $k-1$%
, thus requiring the use of \textbf{stronger} testing conditions, taken
locally over polynomials of degree at most $k-1$, in order to the bound the
consecutive sums of Alpert projections that arise in the paraproduct and
stopping forms in \cite{LaWi} and \cite{SaShUr7}.
\end{itemize}

With these two changes in mind we can now review the main steps in the
standard proof strategy for the interested reader, whom we alert to the fact
that we are using here the formulation of the Lacey-Wick monotonicity lemma
with an error term in \cite{LaWi} and \cite{SaShUr7}, as opposed to the
stronger formulation used in \cite{LaSaShUr3} and \cite{Lac} that exploited
special properties of the Haar basis to hide the error term. As a
consequence, the reader can follow the broad outline of the one-dimensional
proof in \cite{LaSaShUr3} and \cite{Lac}, but handling the error terms as in 
\cite{LaWi} and/or \cite{SaShUr7} (see also \cite{SaShUr9}, \cite{SaShUr10}%
). We ignore the case of common point masses, and refer the reader instead
to \cite{Hyt2}, \cite{LaWi} and \cite{SaShUr10}.

\begin{description}
\item[Step 1] Using the random grids of Nazarov, Treil and Volberg, a
reduction is made to good functions $f$ and $g$, i.e. those whose wavelet
expansions involve only cubes from one grid that are \emph{good} with
respect to the other grid. The orthogonality of Alpert projections plays a
key role here.

\item[Step 2] Using the testing conditions, one further restricts the
supports of $f$ and $g$ to a finite union of large cubes.

\item[Step 3] Then one can implement corona constructions with Calder\'{o}%
n-Zygmund stopping times on the averages of $f$, and with $k$-energy
stopping times in place of the familiar energy stopping times.

\item[Step 4] The sum of inner products in (\ref{inner pdt sum}) is then
grouped into coronas relative to these stopping times and further decomposed
into global, local and error pieces.

\item[Step 5] The error pieces are handled by NTV methods from \cite{NTV4}.

\item[Step 6] The global inner products are controlled by the $k$-Poisson
operator as in \cite{LaWi} and/or \cite{SaShUr7}, which in turn has its norm
inequality controlled by $k$-Poisson testing conditions. This latter result
is proved in the same way as is done for the familiar Poisson operator.

\item[Step 7] The local terms are handled by Lacey's bottom/up stopping time
and recursion as in \cite{Lac}, with error terms from the monotonicity lemma
handled as in \cite{LaWi} and/or \cite{SaShUr7}. The difficulty lies in
using the Nazarov, Treil, Volberg method connecting back to the appropriate
paraproduct terms; at the moment we are unable to control those terms and so
do not have a proof of Conjecture \ref{TP}.
\end{description}

\begin{remark}
There is an analogous conjecture in higher dimensions following the
arguments in \cite{LaWi} and/or \cite{SaShUr7}, but we will not pursue this.
\end{remark}

\subsection{The Monotonicity Lemma}

For $0\leq \alpha <1$ and $m\in \mathbb{R}_{+}$, we recall the $m$-weighted
fractional Poisson integral%
\begin{equation*}
\mathrm{P}_{m}^{\alpha }\left( J,\mu \right) \equiv \int_{\mathbb{R}}\frac{%
\left\vert J\right\vert ^{m}}{\left( \left\vert J\right\vert +\left\vert
y-c_{J}\right\vert \right) ^{m+1-\alpha }}d\mu \left( y\right) ,
\end{equation*}%
where $\mathrm{P}_{1}^{\alpha }\left( J,\mu \right) =\mathrm{P}^{\alpha
}\left( J,\mu \right) $ is the standard Poisson integral.

\begin{lemma}[Monotonicity]
\label{mono}Suppose that$\ I$ and $J$ are cubes in $\mathbb{R}$ such that $%
J\subset 2J\subset I$, and that $\mu $ is a signed measure on $\mathbb{R}$
supported outside $I$. Finally suppose that $T^{\alpha }$ is a standard
fractional singular integral on $\mathbb{R}$ with kernel $K^{\alpha }\left(
x,y\right) =K_{y}^{\alpha }\left( x\right) $, $0<\alpha <1$. Then there is a
positive constant $C_{\alpha }$ such that%
\begin{equation}
\Phi ^{\alpha }\left( J,\mu \right) ^{2}-C_{\alpha }\Psi ^{\alpha }\left(
J,\left\vert \mu \right\vert \right) ^{2}\leq \left\Vert \bigtriangleup
_{J}^{\omega }T^{\alpha }\mu \right\Vert _{L^{2}\left( \omega \right)
}^{2}\lesssim \Phi ^{\alpha }\left( J,\mu \right) ^{2}+C_{\alpha }\Psi
^{\alpha }\left( J,\left\vert \mu \right\vert \right) ^{2},  \label{estimate}
\end{equation}%
where for a measure $\nu $,%
\begin{eqnarray*}
&&\Phi ^{\alpha }\left( J,\nu \right) ^{2}\equiv \left\vert \frac{1}{k!}\int
\left( K_{y}^{\alpha }\right) ^{\left( k\right) }\left( m_{J}\right) d\mu
\left( y\right) \right\vert ^{2}\left\Vert \bigtriangleup _{J}^{\omega
}x^{k}\right\Vert _{L^{2}\left( \omega \right) }^{2}\ , \\
&&\Psi ^{\alpha }\left( J,\nu \right) ^{2}\equiv \left( \frac{\mathrm{P}%
_{k+\delta }^{\alpha }\left( J,\nu \right) }{\left\vert J\right\vert ^{k}}%
\right) ^{2}\left\Vert \left( x-m_{J}^{k}\right) ^{k}\right\Vert
_{L^{2}\left( \mathbf{1}_{J}\omega \right) }^{2}\ , \\
&&\text{where }m_{J}^{k}\in J\text{ satisfies }\left\Vert \left(
x-m_{J}^{k}\right) ^{k}\right\Vert _{L^{2}\left( \mathbf{1}_{J}\omega
\right) }^{2}=\inf_{c\in J}\left\Vert \left( x-c\right) ^{k}\right\Vert
_{L^{2}\left( \mathbf{1}_{J}\omega \right) }^{2}.
\end{eqnarray*}%
and where if $\nu $ is a positive measure, then there are positive constants 
$c,C$ such that%
\begin{equation*}
c\Phi ^{\alpha }\left( J,\nu \right) ^{2}\leq \left( \frac{\mathrm{P}%
_{k}^{\alpha }\left( J,\nu \right) }{\left\vert J\right\vert ^{k}}\right)
^{2}\left\Vert \bigtriangleup _{J}^{\omega }x^{k}\right\Vert _{L^{2}\left(
\omega \right) }^{2}\leq C\Phi ^{\alpha }\left( J,\nu \right) ^{2}\ .
\end{equation*}
\end{lemma}

\begin{remark}
The right hand side of (\ref{estimate}) is what determines the definition of
the $k$-energy condition (\ref{k energy}) used in the stopping time
arguments adapted from \cite{LaWi} and/or \cite{SaShUr7}.
\end{remark}

\begin{proof}[Proof of Lemma \protect\ref{mono}]
The proof is an easy adaptation of the proofs in \cite{LaWi} and \cite%
{SaShUr7} restricted to dimension $n=1$, but using an order $k$ Taylor
expansion instead of an order $1$ expansion on the kernel $\left(
K_{y}^{\alpha }\right) \left( x\right) =K^{\alpha }\left( x,y\right) $. Due
to the importance of this lemma, as explained in the above remark, we give
the short argument.

Let $\left\{ h_{J}^{\omega ,a}\right\} _{a\in \Gamma }$ be an orthonormal
basis of $L_{J;k}^{2}\left( \mu \right) $ consisting of Alpert functions as
above. Now we use the $\left( k+\delta \right) $-smooth Calder\'{o}n-Zygmund
smoothness estimate (\ref{sizeandsmoothness'}), together with Taylor's
formula 
\begin{eqnarray*}
K_{y}^{\alpha }\left( x\right) &=&T\left( K_{y}^{\alpha }\right) \left(
x,c\right) +\frac{1}{k!}\left( K_{y}^{\alpha }\right) ^{\left( k\right)
}\left( \theta \left( x,c\right) \right) \left( x-c\right) ^{k}; \\
\limfunc{Tay}\left( K_{y}^{\alpha }\right) \left( x,c\right) &\equiv
&K_{y}^{\alpha }\left( c\right) +\left( K_{y}^{\alpha }\right) ^{\prime
}\left( c\right) \left( x-c\right) +...+\frac{1}{\left( k-1\right) !}\left(
K_{y}^{\alpha }\right) ^{\left( k-1\right) }\left( c\right) \left(
x-c\right) ^{k-1},
\end{eqnarray*}%
and the vanishing means of the Alpert functions $h_{J}^{\omega ,a}$, to
obtain 
\begin{eqnarray*}
\left\langle T^{\alpha }\mu ,h_{J}^{\omega ,a}\right\rangle _{\omega }
&=&\int \left\{ \int K^{\alpha }\left( x,y\right) h_{J}^{\omega ,a}\left(
x\right) d\omega \left( x\right) \right\} d\mu \left( y\right) =\int
\left\langle K_{y}^{\alpha },h_{J}^{\omega ,a}\right\rangle _{\omega }d\mu
\left( y\right) \\
&=&\int \left\langle K_{y}^{\alpha }\left( x\right) -\limfunc{Tay}\left(
K_{y}^{\alpha }\right) \left( x,m_{J}^{k}\right) ,h_{J}^{\omega ,a}\left(
x\right) \right\rangle _{\omega }d\mu \left( y\right) \\
&=&\int \left\langle \frac{1}{k!}\left( K_{y}^{\alpha }\right) ^{\left(
k\right) }\left( \theta \left( x,m_{J}^{k}\right) \right) \left(
x-m_{J}^{k}\right) ^{k},h_{J}^{\omega ,a}\left( x\right) \right\rangle
_{\omega }d\mu \left( y\right) \\
&=&\left\langle \left[ \int \frac{1}{k!}\left( K_{y}^{\alpha }\right)
^{\left( k\right) }\left( m_{J}\right) d\mu \left( y\right) \right] \left(
x-m_{J}^{k}\right) ^{k},h_{J}^{\omega ,a}\right\rangle _{\omega } \\
&&+\left\langle \left[ \int \frac{1}{k!}\left[ \left( K_{y}^{\alpha }\right)
^{\left( k\right) }\left( \theta \left( x,m_{J}^{k}\right) \right) -\left(
K_{y}^{\alpha }\right) ^{\left( k\right) }\left( m_{J}^{k}\right) \right]
d\mu \left( y\right) \right] \left( x-m_{J}^{k}\right) ^{k},h_{J}^{\omega
,a}\right\rangle _{\omega } \\
&=&\left[ \frac{1}{k!}\int \left( K_{y}^{\alpha }\right) ^{\left( k\right)
}\left( m_{J}\right) d\mu \left( y\right) \right] \left\langle
x^{k},h_{J}^{\omega ,a}\right\rangle _{\omega } \\
&&+\left\langle \left[ \int \frac{1}{k!}\left[ \left( K_{y}^{\alpha }\right)
^{\left( k\right) }\left( \theta \left( x,m_{J}^{k}\right) \right) -\left(
K_{y}^{\alpha }\right) ^{\left( k\right) }\left( m_{J}^{k}\right) \right]
d\mu \left( y\right) \right] \left( x-m_{J}^{k}\right) ^{k},h_{J}^{\omega
,a}\right\rangle _{\omega }
\end{eqnarray*}%
and hence%
\begin{eqnarray*}
&&\left\vert \left\langle T^{\alpha }\mu ,h_{J}^{\omega ,a}\right\rangle
_{\omega }-\left[ \frac{1}{k!}\int \left( K_{y}^{\alpha }\right) ^{\left(
k\right) }\left( m_{J}\right) d\mu \left( y\right) \right] \left\langle
x^{k},h_{J}^{\omega ,a}\right\rangle _{\omega }\right\vert \\
&\leq &\frac{1}{k!}\left\vert \left\langle \left[ \int \sup_{\theta \in
J}\left\vert \left( K_{y}^{\alpha }\right) ^{\left( k\right) }\left( \theta
\right) -\left( K_{y}^{\alpha }\right) ^{\left( k\right) }\left(
m_{J}^{k}\right) \right\vert d\mu \left( y\right) \right] \left\vert
x-m_{J}^{k}\right\vert ^{k},\left\vert h_{J}^{\omega ,a}\right\vert
\right\rangle _{\omega }\right\vert \\
&\lesssim &C_{CZ}\frac{\mathrm{P}_{k+\delta }^{\alpha }\left( J,\left\vert
\mu \right\vert \right) }{\left\vert J\right\vert ^{k}}\left\Vert \left(
x-m_{J}^{k}\right) ^{k}\right\Vert _{L^{2}\left( \mathbf{1}_{J}\omega
\right) }
\end{eqnarray*}%
where in the last line we have used%
\begin{eqnarray*}
&&\int \sup_{\theta \in J}\left\vert \left( K_{y}^{\alpha }\right) ^{\left(
k\right) }\left( \theta \right) -\left( K_{y}^{\alpha }\right) ^{\left(
k\right) }\left( m_{J}^{k}\right) \right\vert d\mu \left( y\right) \\
&\lesssim &C_{CZ}\int \left( \frac{\left\vert J\right\vert }{\left\vert
y-c_{J}\right\vert }\right) ^{\delta }\frac{d\left\vert \mu \right\vert
\left( y\right) }{\left\vert y-c_{J}\right\vert ^{k+1-\alpha }}=C_{CZ}\frac{%
\mathrm{P}_{k+\delta }^{\alpha }\left( J,\mu \right) }{\left\vert
J\right\vert ^{k}}.
\end{eqnarray*}

Thus we have%
\begin{eqnarray*}
\left\Vert \bigtriangleup _{J}^{\omega }T^{\alpha }\mu \right\Vert
_{L^{2}\left( \omega \right) }^{2} &=&\sum_{a\in \Gamma \left( J\right)
}\left\vert \left\langle T^{\alpha }\mu ,h_{J}^{\omega ,a}\right\rangle
_{\omega }\right\vert ^{2} \\
&=&\left\vert \frac{1}{k!}\int \left( K_{y}^{\alpha }\right) ^{\left(
k\right) }\left( m_{J}\right) d\mu \left( y\right) \right\vert
^{2}\sum_{a\in \Gamma \left( J\right) }\left\vert \left\langle
x^{k},h_{J}^{\omega ,a}\right\rangle _{\omega }\right\vert ^{2} \\
&&+O\left( \frac{\mathrm{P}_{k+\delta }^{\alpha }\left( J,\mu \right) }{%
\left\vert J\right\vert ^{k}}\right) ^{2}\left\Vert \left(
x-m_{J}^{k}\right) ^{k}\right\Vert _{L^{2}\left( \mathbf{1}_{J}\omega
\right) }^{2},
\end{eqnarray*}%
and hence%
\begin{eqnarray*}
&&c_{1}\left( \frac{1}{k!}\int \left( K_{y}^{\alpha }\right) ^{\left(
k\right) }\left( m_{J}\right) d\mu \left( y\right) \right) ^{2}\left\Vert
\bigtriangleup _{J}^{\omega }x^{k}\right\Vert _{L^{2}\left( \omega \right)
}^{2}-C_{2}\left( \frac{\mathrm{P}_{k+\delta }^{\alpha }\left( J,\left\vert
\mu \right\vert \right) }{\left\vert J\right\vert ^{k}}\right)
^{2}\left\Vert \left( x-m_{J}^{k}\right) ^{k}\right\Vert _{L^{2}\left( 
\mathbf{1}_{J}\omega \right) }^{2} \\
&\leq &\left\Vert \bigtriangleup _{J}^{\omega }T^{\alpha }\mu \right\Vert
_{L^{2}\left( \omega \right) }^{2} \\
&\leq &C_{1}\left( \frac{1}{k!}\int \left( K_{y}^{\alpha }\right) ^{\left(
k\right) }\left( m_{J}\right) d\mu \left( y\right) \right) ^{2}\left\Vert
\bigtriangleup _{J}^{\omega }x^{k}\right\Vert _{L^{2}\left( \omega \right)
}^{2}+C_{2}\left( \frac{\mathrm{P}_{k+\delta }^{\alpha }\left( J,\left\vert
\mu \right\vert \right) }{\left\vert J\right\vert ^{k}}\right)
^{2}\left\Vert \left( x-m_{J}^{k}\right) ^{k}\right\Vert _{L^{2}\left( 
\mathbf{1}_{J}\omega \right) }^{2}\ ,
\end{eqnarray*}%
where%
\begin{equation*}
\left\vert \frac{1}{k!}\int \left( K_{y}^{\alpha }\right) ^{\left( k\right)
}\left( m_{J}\right) d\mu \left( y\right) \right\vert \approx \frac{\mathrm{P%
}_{k}^{\alpha }\left( J,\mu \right) }{\left\vert J\right\vert ^{k}}\ .
\end{equation*}
\end{proof}

\subsection{Comparison of $k$-energy and the usual $1$-energy}

We can write%
\begin{equation*}
\left( \frac{\mathrm{P}_{k}^{\alpha }\left( J,\mathbf{1}_{I}\sigma \right) }{%
\left\vert J\right\vert ^{k}}\right) ^{2}\left\Vert \left(
x-m_{J}^{k}\right) ^{k}\right\Vert _{L^{2}\left( \mathbf{1}_{J}\omega
\right) }^{2}=\mathrm{P}_{k}^{\alpha }\left( J,\mathbf{1}_{I}\sigma \right)
^{2}\left\Vert \left( \frac{x-m_{J}^{k}}{\left\vert J\right\vert }\right)
^{k}\right\Vert _{L^{2}\left( \mathbf{1}_{J}\omega \right) }^{2}\ ,
\end{equation*}%
and clearly we have the inequalities%
\begin{eqnarray*}
\mathrm{P}_{k}^{\alpha }\left( J,\mathbf{1}_{I}\sigma \right) &=&\int_{%
\mathbb{R}}\frac{\left\vert J\right\vert ^{k}}{\left( \left\vert
J\right\vert +\left\vert y-c_{J}\right\vert \right) ^{k+1-\alpha }}d\sigma
\left( y\right) \\
&=&\int_{\mathbb{R}}\left( \frac{\left\vert J\right\vert }{\left( \left\vert
J\right\vert +\left\vert y-c_{J}\right\vert \right) }\right) ^{k-\ell }\frac{%
\left\vert J\right\vert ^{\ell }}{\left( \left\vert J\right\vert +\left\vert
y-c_{J}\right\vert \right) ^{\ell +1-\alpha }}d\sigma \left( y\right) \\
&\leq &\int_{\mathbb{R}}\frac{\left\vert J\right\vert ^{\ell }}{\left(
\left\vert J\right\vert +\left\vert y-c_{J}\right\vert \right) ^{\ell
+1-\alpha }}d\sigma \left( y\right) =\mathrm{P}_{\ell }^{\alpha }\left( J,%
\mathbf{1}_{I}\sigma \right) ,
\end{eqnarray*}%
and 
\begin{eqnarray*}
\left\Vert \left( \frac{x-m_{J}^{k}}{\left\vert J\right\vert }\right)
^{k}\right\Vert _{L^{2}\left( \mathbf{1}_{J}\omega \right) }^{2} &\leq
&\left\Vert \left( \frac{x-m_{J}^{\ell }}{\left\vert J\right\vert }\right)
^{k}\right\Vert _{L^{2}\left( \mathbf{1}_{J}\omega \right)
}^{2}=\int_{J}\left( \frac{x-m_{J}^{\ell }}{\left\vert J\right\vert }\right)
^{2k}d\omega \left( x\right) \\
&\leq &\int_{J}\left( \frac{x-m_{J}^{\ell }}{\left\vert J\right\vert }%
\right) ^{2\ell }d\omega \left( x\right) =\left\Vert \left( \frac{%
x-m_{J}^{\ell }}{\left\vert J\right\vert }\right) ^{\ell }\right\Vert
_{L^{2}\left( \mathbf{1}_{J}\omega \right) }^{2},
\end{eqnarray*}%
for $1\leq \ell \leq k$ since $\left\vert \frac{x-m_{J}^{\ell }}{\left\vert
J\right\vert }\right\vert \leq 1$, and as a consequence we obtain%
\begin{equation*}
\mathcal{E}_{2,k}^{\alpha }\leq \mathcal{E}_{2,\ell }^{\alpha },\text{\ \ \
\ \ for }1\leq \ell \leq k.
\end{equation*}

\subsection{An example with $\mathcal{A}_{2}^{0}<\infty $, $\mathcal{E}%
_{2,k}^{0}<\infty $ and $\mathcal{E}_{2,1}^{0}=\infty $}

Define intervals%
\begin{equation*}
I_{j}\equiv \left[ \frac{1}{4^{j}}-\frac{1}{\sqrt{j}4^{j}},\frac{1}{4^{j}}%
\right] ,\ \ \ \ \ j=2,3,4,...
\end{equation*}%
and for $\varepsilon >0$ define measures%
\begin{eqnarray*}
d\sigma \left( y\right) &=&\delta _{0}+\sum_{j=2}^{\infty }\frac{4^{j}}{j^{%
\frac{1}{2}+\varepsilon }}\mathbf{1}_{I_{j}}\left( y\right) dy, \\
d\omega \left( x\right) &=&\sum_{j=2}^{\infty }\frac{j^{\frac{1}{2}}}{4^{j}}%
\mathbf{1}_{I_{j}}\left( x\right) dx,
\end{eqnarray*}%
so that%
\begin{eqnarray*}
\left\vert I_{j}\right\vert &=&\frac{1}{\sqrt{j}4^{j}}, \\
\left\vert I_{j}\right\vert _{\sigma } &=&\frac{4^{j}}{j^{\frac{1}{2}%
+\varepsilon }}\frac{1}{\sqrt{j}4^{j}}=\frac{1}{j^{1+\varepsilon }}, \\
\left\vert I_{j}\right\vert _{\omega } &=&\frac{j^{\frac{1}{2}}}{4^{j}}\frac{%
1}{\sqrt{j}4^{j}}=\frac{1}{4^{2j}}.
\end{eqnarray*}%
Then we compute that%
\begin{eqnarray*}
\frac{\left\vert I_{j}\right\vert _{\sigma }\left\vert I_{j}\right\vert
_{\omega }}{\left\vert I_{j}\right\vert ^{2}} &=&\frac{\frac{1}{%
j^{1+\varepsilon }}\frac{1}{4^{2j}}}{\left( \frac{1}{\sqrt{j}4^{j}}\right)
^{2}}=\frac{1}{j^{\varepsilon }}, \\
\frac{\left\vert \left[ 0,\frac{1}{4^{j}}\right] \right\vert _{\sigma
}\left\vert \left[ 0,\frac{1}{4^{j}}\right] \right\vert _{\omega }}{%
\left\vert \left[ 0,\frac{1}{4^{j}}\right] \right\vert ^{2}} &=&\frac{\left(
1+\sum_{i=j}^{\infty }\left\vert I_{i}\right\vert _{\sigma }\right) \left(
\sum_{i=j}^{\infty }\frac{1}{4^{2i}}\right) }{\left( \frac{1}{4^{j}}\right)
^{2}}\approx 1,
\end{eqnarray*}%
and in fact it can be verified that $\mathcal{A}_{2}\left( \sigma ,\omega
\right) \approx 1$. We also have%
\begin{equation*}
\left\Vert \frac{x-m_{I_{j}}^{1}}{\left\vert I_{j}\right\vert }\right\Vert
_{L^{2}\left( \mathbf{1}_{I_{j}}\omega \right) }^{2}=\int_{I_{j}}\left( 
\frac{x-m_{I_{j}}^{1}}{\left\vert I_{j}\right\vert }\right) ^{2}d\omega
\left( x\right) \approx \left\vert I_{j}\right\vert _{\omega }=\frac{1}{%
4^{2j}}
\end{equation*}%
and writing $\mathrm{P}=\mathrm{P}_{1}^{0}$ with $I=\left[ 0,1\right] $, we
have%
\begin{eqnarray*}
\mathrm{P}\left( I_{j},\mathbf{1}_{I}\sigma \right) &=&\int_{I}\frac{%
\left\vert I_{j}\right\vert }{\left[ \left\vert I_{j}\right\vert +\left\vert
y-c_{I_{j}}\right\vert \right] ^{2}}d\sigma \left( y\right) \approx \frac{%
\left\vert I_{j}\right\vert }{\left[ \left\vert I_{j}\right\vert +\left\vert
0-c_{I_{j}}\right\vert \right] ^{2}}+\int_{I_{j}}\frac{\left\vert
I_{j}\right\vert }{\left[ \left\vert I_{j}\right\vert +\left\vert
y-c_{I_{j}}\right\vert \right] ^{2}}d\sigma \left( y\right) \\
&\approx &\frac{\frac{1}{\sqrt{j}4^{j}}}{\left[ \frac{1}{4^{j}}\right] ^{2}}+%
\frac{\left\vert I_{j}\right\vert _{\sigma }}{\left\vert I_{j}\right\vert }=%
\frac{4^{j}}{\sqrt{j}}+\frac{\frac{1}{j^{1+\varepsilon }}}{\frac{1}{\sqrt{j}%
4^{j}}}=\frac{4^{j}}{j^{\frac{1}{2}}}+\frac{4^{j}}{j^{\frac{1}{2}%
+\varepsilon }}\approx \frac{4^{j}}{j^{\frac{1}{2}}},
\end{eqnarray*}%
where we note that the delta mass $\delta _{0}$ in $\sigma $ contributes the
dominant term. From this we compute%
\begin{equation*}
\mathcal{E}_{2,1}^{\alpha }\geq \frac{1}{\left\vert I\right\vert _{\sigma }}%
\sum_{j=2}^{\infty }\mathrm{P}\left( I_{j},\mathbf{1}_{I}\sigma \right)
^{2}\left\Vert \frac{x-m_{I_{j}}^{1}}{\left\vert I_{j}\right\vert }%
\right\Vert _{L^{2}\left( \mathbf{1}_{I_{j}}\omega \right) }^{2}\approx
\sum_{j=2}^{\infty }\left( \frac{4^{j}}{j^{\frac{1}{2}}}\right) ^{2}\frac{1}{%
4^{2j}}=\sum_{j=2}^{\infty }\frac{1}{j}=\infty .
\end{equation*}

On the other hand, writing $\mathrm{P}_{k}=\mathrm{P}_{k}^{0}$ we have%
\begin{eqnarray*}
\mathrm{P}_{k}\left( I_{j},\mathbf{1}_{I}\sigma \right) &=&\int_{I}\left( 
\frac{\left\vert I_{j}\right\vert }{\left( \left\vert I_{j}\right\vert
+\left\vert y-c_{I_{j}}\right\vert \right) }\right) ^{k-1}\frac{\left\vert
I_{j}\right\vert }{\left[ \left\vert I_{j}\right\vert +\left\vert
y-c_{I_{j}}\right\vert \right] ^{2}}d\sigma \left( y\right) \\
&\approx &\left( \frac{\left\vert I_{j}\right\vert }{\left( \left\vert
I_{j}\right\vert +\left\vert 0-c_{I_{j}}\right\vert \right) }\right) ^{k-1}%
\frac{\left\vert I_{j}\right\vert }{\left[ \left\vert I_{j}\right\vert
+\left\vert 0-c_{I_{j}}\right\vert \right] ^{2}}+\int_{I_{j}}\frac{%
\left\vert I_{j}\right\vert }{\left[ \left\vert I_{j}\right\vert +d\left(
y,I_{j}\right) \right] ^{2}}d\sigma \left( y\right) \\
&\approx &\left( \frac{\frac{1}{\sqrt{j}4^{j}}}{\frac{1}{4^{j}}}\right)
^{k-1}\frac{\frac{1}{\sqrt{j}4^{j}}}{\left[ \frac{1}{4^{j}}\right] ^{2}}+%
\frac{\left\vert I_{j}\right\vert _{\sigma }}{\left\vert I_{j}\right\vert }%
\approx \frac{4^{j}}{j^{\frac{k}{2}}}+\frac{\frac{1}{j^{1+\varepsilon }}}{%
\frac{1}{\sqrt{j}4^{j}}}\approx \frac{4^{j}}{j^{\frac{1}{2}+\varepsilon }},
\end{eqnarray*}%
where when $k\geq 2$, the dominant term arises from that part of $\sigma $
supported on $I_{j}$, and hence is significantly smaller than $\mathrm{P}%
\left( I_{j},\mathbf{1}_{I}\sigma \right) $. Then for $k\geq 2$ we compute%
\begin{equation*}
\frac{1}{\left\vert I\right\vert _{\sigma }}\sum_{j=2}^{\infty }\mathrm{P}%
_{k}\left( I_{j},\mathbf{1}_{I}\sigma \right) ^{2}\left\Vert \left( \frac{%
x-m_{I_{j}}^{k}}{\left\vert I_{j}\right\vert }\right) ^{k}\right\Vert
_{L^{2}\left( \mathbf{1}_{I_{j}}\omega \right) }^{2}\lesssim
\sum_{j=2}^{\infty }\left( \frac{4^{j}}{j^{\frac{1}{2}+\varepsilon }}\right)
^{2}\frac{1}{4^{2j}}=\sum_{j=2}^{\infty }\frac{1}{j^{1+2\varepsilon }}%
<\infty ,
\end{equation*}%
and it can be shown that in fact $\mathcal{E}_{2,k}^{\alpha }<\infty $ by
considering arbitrary decompositions $I=\dot{\cup}_{r=1}^{\infty }I_{r}$.

\subsection{A Calder\'{o}n-Zygmund operator satisfying testing conditions 
\label{CZ test}}

We do not have an example of a weight pair $\left( \sigma ,\omega \right) $,
and a familiar Calder\'{o}n-Zygmund operator $T$, to which Conjecture \ref%
{TP} applies, and to which the known $T1$-type theorems fail to apply. Our
purpose here is to instead construct a rather artificial example to
demonstrate that Conjecture \ref{TP} is at least `logically' different than
the known $T1$-type theorems in \cite{LaWi} and \cite{SaShUr7}.

For this we first construct a dyadic operator $T_{\left( \sigma ,\omega
\right) }^{\limfunc{dy}}$ that always satisfies testing conditions for a
given weight pair $\left( \sigma ,\omega \right) $. For an arbitrary weight
pair $\left( \sigma ,\omega \right) $ define%
\begin{equation*}
T_{\left( \sigma ,\omega \right) }^{\limfunc{dy}}f\equiv \sum_{I\in \mathcal{%
D}}\sum_{1\leq i,j\leq 2}\left\langle f,a_{I}^{\sigma ,i}\right\rangle
_{L^{2}\left( \sigma \right) }a_{I}^{\omega ,j}\ ,
\end{equation*}%
where $\left\{ a_{I}^{\sigma ,1},a_{I}^{\sigma ,2}\right\} _{I\in \mathcal{D}%
}$ and $\left\{ a_{I}^{\omega ,1},a_{I}^{\omega ,2}\right\} _{I\in \mathcal{D%
}}$ are Alpert bases for $L^{2}\left( \sigma \right) $ and $L^{2}\left(
\omega \right) $ respectively. Then for any interval $Q\in \mathcal{D}$ and
polynomial $p$ of degree at most $1$ we have%
\begin{eqnarray*}
\int_{Q}\left\vert T_{\left( \sigma ,\omega \right) }^{\limfunc{dy}}\mathbf{1%
}_{Q}p\right\vert ^{2}d\omega &=&\int_{Q}\left\vert \sum_{I\in \mathcal{D}%
}\sum_{1\leq i,j\leq 2}\left\langle \mathbf{1}_{Q}p,a_{I}^{\sigma
,i}\right\rangle _{L^{2}\left( \sigma \right) }a_{I}^{\omega ,j}\right\vert
^{2}d\omega \\
&=&\sum_{I\in \mathcal{D}}\sum_{1\leq j\leq 2}\int_{Q}\left\vert \sum_{1\leq
i\leq 2}\left\langle \mathbf{1}_{Q}p,a_{I}^{\sigma ,i}\right\rangle
_{L^{2}\left( \sigma \right) }\right\vert ^{2}\left\vert a_{I}^{\omega
,j}\right\vert ^{2}d\omega \\
&=&\sum_{I\in \mathcal{D}:\ I\varsupsetneqq Q}\sum_{1\leq j\leq
2}\int_{Q}\left\vert \widehat{\mathbf{1}_{Q}p}\left( I\right) \right\vert
^{2}\left\vert a_{I}^{\omega ,j}\right\vert ^{2}d\omega \\
&=&\sum_{I\in \mathcal{D}:\ I\varsupsetneqq Q}\left\vert \widehat{\mathbf{1}%
_{Q}p}\left( I\right) \right\vert ^{2}\leq \int_{Q}\left\vert p\right\vert
^{2}d\sigma .
\end{eqnarray*}%
Similarly we have the dual testing condition.

To construct a Calder\'{o}n-Zygmund operator $T$ that satisfies the testing
conditions for $\left( \sigma ,\omega \right) $ and whose kernel is $\left(
k+\delta \right) $-smooth as in (\ref{sizeandsmoothness'}) for any $k+\delta
>1$, we choose a subgrid $\mathcal{D}^{\prime }$ of $\mathcal{D}$ satisfying:%
\begin{equation}
\text{if }I,J\in \mathcal{D}^{\prime }\text{ satisfy }\left\vert
I\right\vert \leq \left\vert J\right\vert \text{ and }\frac{11}{9}I\cap 
\frac{11}{9}J\neq \emptyset \text{, then }\frac{10}{9}I\subset J\text{,}
\label{D'}
\end{equation}%
and define functions $a_{I}^{\sigma }$ and $a_{I}^{\omega }$ to satisfy the
following conditions with $\mu $ equal $\sigma $ and $\omega $: 
\begin{eqnarray*}
\func{Supp}a_{I}^{\mu } &\subset &\frac{11}{9}I, \\
0 &\leq &a_{I}^{\mu }\left( x\right) \leq c\left\vert I\right\vert ^{-\frac{1%
}{2}}\text{ for }x\in \mathbb{R}, \\
a_{I}^{\mu }\left( x\right) &=&c\text{ for }x\in \frac{10}{9}I, \\
\int \left\vert a_{I}^{\mu }\right\vert ^{2}d\mu &=&1, \\
\int_{\mathbb{R}}xa_{I}^{\mu }\left( x\right) d\mu \left( x\right) &=&\int_{%
\mathbb{R}}a_{I}^{\mu }\left( x\right) d\mu \left( x\right) =0, \\
\left\vert \frac{d^{\ell }}{dx^{\ell }}a_{I}^{\mu }\left( x\right)
\right\vert &\leq &C_{\ell }\left\vert I\right\vert ^{-\frac{1}{2}-\ell },\
\ \ \ \ \ell \geq 0.
\end{eqnarray*}%
Note that these functions are \emph{not} Alpert functions, but that they 
\emph{do} form an orthonormal collection in $L^{2}\left( \mu \right) $. Then
we set 
\begin{equation*}
Tf=\sum_{I\in \mathcal{D}^{\prime }}\left\langle f,a_{I}^{\sigma
}\right\rangle _{L^{2}\left( \sigma \right) }a_{I}^{\omega }
\end{equation*}%
and note that $T$ is a $\left( k+\delta \right) $-smooth Calder\'{o}%
n-Zygmund operator (see below) that satisfies the testing conditions for the
weight pair $\left( \sigma ,\omega \right) $ by the same argument above that
established the testing conditions for the dyadic operator $T_{\left( \sigma
,\omega \right) }^{\limfunc{dy}}$. If we take the weight pair $\left( \sigma
,\omega \right) $ constructed above, then Conjecture \ref{TP} gives the
boundedness of $T$ from $L^{2}(\sigma )$ to $L^{2}(\omega )$.

\begin{lemma}
The operator $T$ defined above is a $\left( k+\delta \right) $-smooth $0$%
-fractional standard Calder\'{o}n-Zymund operator for all $k+\delta >1$.
\end{lemma}

\begin{proof}
Note that the kernel is $K(x,y)=\sum_{I\in \mathcal{D}^{\prime
}}a_{I}^{\omega }(x)a_{I}^{\sigma }(y)$. Assume that $\left\vert
x-y\right\vert \simeq \frac{11}{9}2^{-\ell }$ (i.e. $\frac{11}{9}2^{-\ell
}\lesssim \left\vert x-y\right\vert \lesssim \frac{11}{9}2^{-\ell +1}$). The
sum in the expression of $K(x,y)$ only contains those terms which correspond
to the intervals in $\mathcal{D}^{\prime }$ for which $x,y\in \frac{11}{9}I$%
. The smallest such interval has length $2^{-\ell }$. Also, for $%
m=0,1,\ldots ,\ell -3$ there are at most two intervals of length $2^{-\ell
+m}$ for which $a_{I}^{\omega }(x)a_{I}^{\sigma }(y)\neq 0$. Thus there are
at most $2(\ell -2)$ terms in the sum for $K(x,y)$.

Also, using the properties of $a_{I}^{\omega }$ and $a_{I}^{\sigma }$ there
holds: 
\begin{eqnarray*}
\left\vert a_{I}^{\omega }(x)a_{I}^{\sigma }(y)\right\vert &\lesssim
&\left\vert I\right\vert ^{-\frac{1}{2}}\left\vert I\right\vert ^{-\frac{1}{2%
}}=\left\vert I\right\vert ^{-1}, \\
\left\vert \frac{\partial }{\partial x}a_{I}^{\omega }(x)a_{I}^{\sigma
}(y)\right\vert &\lesssim &\left\vert I\right\vert ^{-\frac{1}{2}%
-1}\left\vert I\right\vert ^{-\frac{1}{2}}=\left\vert I\right\vert ^{-2}, \\
\left\vert \frac{\partial }{\partial y}a_{I}^{\omega }(x)a_{I}^{\sigma
}(y)\right\vert &\lesssim &\left\vert I\right\vert ^{-\frac{1}{2}}\left\vert
I\right\vert ^{-\frac{1}{2}-1}=\left\vert I\right\vert ^{-2}, \\
\left\vert \nabla ^{k}a_{I}^{\omega }(x)a_{I}^{\sigma }(y)\right\vert &\leq
&C_{k}\left\vert I\right\vert ^{-1-k},
\end{eqnarray*}%
for all $k\geq 1$, and where $\nabla =\left( \frac{\partial }{\partial x},%
\frac{\partial }{\partial y}\right) $.

We now use these estimates to show that the kernel $K(x,y)$ is a $\left(
k+\delta \right) $-smooth $0$-fractional CZ kernel for all $k+\delta >1$. In
fact it is well known (using the mean value theorem) that it suffices to show%
\begin{equation*}
\left\vert \nabla K(x,y)\right\vert \leq C_{k}\frac{1}{\left\vert
x-y\right\vert ^{k+1}},\ \ \ \ \ \text{for all }k\geq 0.
\end{equation*}%
Recalling that $\left\vert x-y\right\vert \simeq \frac{11}{9}2^{-\ell }$ we
have for $k\geq 0$, 
\begin{eqnarray*}
\left\vert \nabla ^{k}K(x,y)\right\vert &\leq &C_{k}\sum_{I\in \mathcal{D}%
^{\prime }}\sum_{i+j=k}\left\vert \left( \nabla _{x}^{i}a_{I}^{\omega
}(x)\right) \left( \nabla _{y}^{j}a_{I}^{\sigma }(y)\right) \right\vert \leq
C_{k}\sum_{I\in \mathcal{D}^{\prime }}\left\vert I\right\vert ^{-k-1} \\
&\leq &C_{k}\sum_{m=0}^{\ell -3}2^{\left( k+1\right) (\ell -m)}\approx
C_{k}2^{\ell \left( k+1\right) }\approx C_{k}\frac{1}{\left\vert
x-y\right\vert ^{k+1}}.
\end{eqnarray*}
\end{proof}

\begin{remark}
However, we can write this measure $\sigma $ as $\delta _{0}+\widetilde{%
\sigma }$ and note that the weight pair $\left( \widetilde{\sigma },\omega
\right) $ satisfies the usual energy condition (the $1$-energy condition),
and so $T$ is bounded from $L^{2}(\widetilde{\sigma })$ to $L^{2}(\omega )$
by results in either \cite{LaWi} or \cite{SaShUr7}, and hence 
\begin{equation*}
\int \left\vert Tf\sigma \right\vert ^{2}d\omega \lesssim \int \left\vert Tf%
\widetilde{\sigma }\right\vert ^{2}d\omega +\int \left\vert Tf\delta
_{0}\right\vert ^{2}d\omega \lesssim \int \left\vert f\right\vert ^{2}d%
\widetilde{\sigma }+\left\vert f\left( 0\right) \right\vert ^{2}\int
\left\vert \sum_{I\in \mathcal{D}^{\prime }}a_{I}^{\sigma }\left( 0\right)
a_{I}^{\omega }\right\vert ^{2}d\omega .
\end{equation*}%
Since $\left\vert f\left( 0\right) \right\vert ^{2}=\int \left\vert
f\right\vert ^{2}d\delta _{0}\leq \int \left\vert f\right\vert ^{2}d\sigma $%
, it now remains only to prove that%
\begin{equation*}
\int \left\vert \sum_{I\in \mathcal{D}^{\prime }}a_{I}^{\sigma }\left(
0\right) a_{I}^{\omega }\right\vert ^{2}d\omega =\sum_{I\in \mathcal{D}%
^{\prime }}\left\vert a_{I}^{\sigma }\left( 0\right) \right\vert
^{2}\lesssim 1,
\end{equation*}%
which in turn holds simply because there is at most one $I\in \mathcal{D}%
^{\prime }$ containing the origin, and for such an interval we have that $%
\left\vert a_{I}^{\sigma }\left( 0\right) \right\vert \lesssim 1$ due to the
presence of the unit point mass at the origin, a much better bound than the
general bound $c\left\vert I\right\vert ^{-\frac{1}{2}}$. Indeed, the only
intervals $I\in \mathcal{D}$ that contain the origin are the intervals $%
K_{j}=\left[ 0,4^{-j}\right) $ for some integer $j$, and by (\ref{D'}) there
is at most one of these in $\mathcal{D}^{\prime }$ containing the origin.
Thus we see that the Calder\'{o}n-Zygmund operator $T$ constructed above is
very artificial. Nonetheless, it does show that Conjecture \ref{TP} contains
boundedness results not included in \cite{LaWi} and \cite{SaShUr7}, and
potentially some not so trivial as that above as well.
\end{remark}

\end{document}